\documentclass[letterpaper,11.5pt]{amsart}

\usepackage{amssymb}
\usepackage{enumerate}
\usepackage[margin=1in]{geometry}
\allowdisplaybreaks

\newtheorem{Thm}{Theorem}[section]

\newtheorem{corollary}[Thm]{Corollary}
\newtheorem{lemma}[Thm]{Lemma}
\newtheorem{proposition}[Thm]{Proposition}

\newtheorem{theorem}[Thm]{Theorem}

\usepackage{color, soul}

\theoremstyle{definition}
\newtheorem{definition}[Thm]{Definition}
\newtheorem{example}[Thm]{Example}
\newtheorem{problem}[Thm]{Problem}

\theoremstyle{remark}
\newtheorem{remark}[Thm]{Remark}

\numberwithin{equation}{section}

\DeclareMathOperator*{\spn}{\text{\rm span}}

\def\ldots{\mathinner{\ldotp\ldotp\ldotp}}
\def\ldots{\mathinner{\cdotp\cdotp\cdotp}}

\def \H{\mathbb H}

\def\cW{\mathcal{W}}

\def\G{\mathcal{G}}
\def \N{\mathbb N}
\def \Z{\mathbb Z}
\def \R{\mathbb R}

\def \C{\mathbb C}
\def \p{\varphi}
\def \HN{{\mathbb H}_{N}}

\def\<{\langle}
\def\>{\rangle}
\def \beq{\begin{eqnarray*}}
\def \eeq{\end{eqnarray*}}
\newcommand{\norm}[1]{\left\| #1 \right\|}

\DeclareMathOperator*{\ran}{ran}
\DeclareMathOperator*{\rank}{rank}
\renewcommand{\phi}{\varphi}
\renewcommand{\hl}[1]{#1\index{#1}}

\makeindex

\begin{document}

\title[Frames and their Applications]
{A Brief Introduction to Hilbert Space Frame Theory
and its Applications 
}
\author[P.G. Casazza and R.G. Lynch
 ]{Peter G. Casazza and Richard G. Lynch}
\address{Department of Mathematics, University
of Missouri, Columbia, MO 65211-4100}

\thanks{The authors were supported by NSF DMS 1307685; and  NSF ATD 1042701 and 1321779; AFOSR  DGE51:  FA9550-11-1-0245}

\email{casazzap@missouri.edu, rglz82@mail.missouri.edu}


\begin{abstract}
This is a short introduction to Hilbert space 
frame theory and its applications for those outside the
area who want to enter the subject.  We will emphasize finite frame
theory since it is the easiest way to get into the subject.
\end{abstract}

\maketitle
\tableofcontents

\newpage
\section{Reading List}

 For a more complete treatment of frame
theory we recommend the books of Han, Kornelson, Larson, and Weber
\cite{HKLW}, Christensen \cite{Ch},
the book of Casazza and Kutyniok \cite{CK12},
the tutorials of Casazza \cite{CA,C1} and the memoir 
of Han and Larson \cite{HL}.  For a
complete treatment of
frame theory in time-frequency analysis we recommend the book
of Gr\"ochenig \cite{G1}.  For an introduction to frame
theory and filter banks plus applications to engineering
we recommend 
Kova\v{c}evi\'{c} and Chebira \cite{K}.  Also, a wealth of information can
be found at the Frame Research Center's website \cite{FRC}.


\section{The Basics of Hilbert Space Theory}

Given a positive integer $N$, we denote by $\H^N$ the real or complex
Hilbert space of dimension $N$.  This is either $\R^N$ or $\C^N$ with
the {\bfseries {inner product}\index{inner product}} given by
\[ 
\langle x,y\rangle = \sum_{i=1}^Na_i\overline{b_i}
\]
for $x=(a_1,a_2,\ldots,a_N)$ and 
$y=(b_1,b_2,\ldots,b_N)$ and the {\bf {norm}\index{norm}} of a vector $x$ is
\[ \|x\|^2 = \langle x,x\rangle.\]
For $x,y \in \H^N$, $\|x-y\|$ is the {\bfseries distance\index{distance}} from the vector $x$ to the
vector $y$.  For future reference, note that in the real case,
\begin{align*}
 \|x-y\|^2 &= \langle x-y,x-y\rangle \\
&= \langle x,x\rangle - \langle x,y\rangle-  \langle y,x \rangle + \langle 
y,y\rangle\\
&= \|x\|^2 - 2 \langle x,y\rangle + \|y\|^2
\end{align*}
and in the complex case we have
\[ \langle x,y\rangle + \langle y,x\rangle = \langle x,y\rangle + \overline{
\langle x,y\rangle} = 2 Re\, \langle x,y\rangle,\]
where $Re \, c$ denotes the real part of the complex number c.
We will concentrate on finite dimensional Hilbert spaces 
 since it is the
easiest way to get started on the subject of frames.  Most of these
results hold for infinite dimensional Hilbert spaces and at the end we
will look at the infinite dimensional case.

The next lemma contains a {\bf standard trick} for calculations.

\begin{lemma}
If $x\in \H^N$ and $\langle x,y\rangle =0$ for all $y \in \H^N$ then
$x=0$.
\end{lemma}

\begin{proof}
Letting $y=x$ we have
\[ 0 = \langle x,y\rangle = \langle x,x\rangle = \|x\|^2,\]
and so $x=0$.
\end{proof}

\begin{definition}
A set of vectors
$\{e_i\}_{i=1}^M$ in $\mathbb{H}^N$ is called:
\begin{enumerate}
\item {\bf linearly independent\index{linearly independent}} if for any scalars $\{a_i\}_{i=1}^M$,
\[ \sum_{i=1}^M a_ie_i=0 \quad  \Rightarrow \quad a_i=0,\mbox{ for all }i=1,2,\ldots,M.\]
Note that this requires $e_i\not= 0$ for all $i=1,2,\ldots,M$.

\item {\bf complete\index{complete}} (or a {\bf spanning set\index{spanning set}}) if 
\[ \overline{\spn} \{e_i\}_{i=1}^M = \H^N. \]

\item {\bf orthogonal\index{orthogonal}} if for all $i\not= j$, $\langle e_i,e_j\rangle=0.$

\item {\bf orthornormal\index{orthonormal}} if it is orthogonal and unit norm.

\item  an {\bf orthonormal basis\index{orthonormal!basis}} if it is complete and orthonormal.
\end{enumerate}
\end{definition}

The following is immediate from the definitions.

\begin{proposition}
If $\{e_i\}_{i=1}^N$ is an orthonormal basis for $\H^N$, then for every
$x\in \H$ we have
\[ x = \sum_{i=1}^N\langle x,e_i\rangle e_i.\]
\end{proposition}

From the previous proposition, we can immediately deduce an essential identity called Parseval's Identity\index{Parseval's Identity}.

\begin{proposition}[Parseval's Identity]\label{parsevalid}
If $\{e_i\}_{i=1}^N$ is an orthonormal basis for $\H^N$, then for every
$x\in \H^N$, we have
\[ \|x\|^2 = \sum_{i=1}^N|\langle x,e_i\rangle|^2.\]
\end{proposition}

Some more basic identities and inequalities for Hilbert space that are frequently used are contained in the next proposition.

\begin{proposition}
Let $x,y\in \H^N$.
\begin{enumerate}
\item  {\bf Cauchy-Schwarz Inequality\label{Cauchy-Schwarz Inequality}}:
\[ |\langle x,y\rangle| \le \|x\|\|y\|,\]
with equality if and only if $x=cy$ for some constant $c$.

\item  {\bf Triangle Inequality\index{Triangle Inequality}}:

\[ \|x+y\|\le \|x\|+\|y\|.\]

\item  {\bf Polarization Identity\index{Polarization Identity}}: Assuming $\mathbb{H}^N$ is real,
\[ \langle x,y\rangle = \frac{1}{4}\left [ \|x+y\|^2 -\|x-y\|^2 
\right ].\]
If $\mathbb{H}^N$ is complex, then
{\hfil\begin{align*}
\langle x,y\rangle = \frac{1}{4}\big [ \|x+y\|^2 -\|x-y\|^2 + i\|x+iy\|^2 -i\|x-iy\|^2
\big ].
\end{align*}}

\item  {\bf Pythagorean Theorem\index{Pythagorean Theorem}}:  Given pairwise orthogonal
vectors $\{x_i\}_{i=1}^M$,
\[ \left \| \sum_{i=1}^Mx_i\right \|^2 = \sum_{i=1}^M\|x_i\|^2.\]
\end{enumerate}
\end{proposition}

\begin{proof}
(1)  The inequality is trivial if $y = 0$. If $y \neq 0$, we may assume
$\|y\|=1$ by dividing through the inequality with $\|y\|$.  Now we compute:
\begin{align*}
0 &< \|x-\langle x,y\rangle y\|^2\\
&= \|x\|^2 - 2 \langle x,y\rangle \overline{\langle x,y\rangle} +|\langle x,y\rangle|^2
\|y\|^2\\
&= \|x\|^2 - |\langle x,y\rangle|^2\\
&= \|x\|^2 \|y\|^2 - |\langle x,y\rangle|^2.
\end{align*}
Note that the strict inequality would be equality if $x = cy$.

(2)  Applying (1) to obtain the second inequality:
\begin{align*}
\|x+y\|^2 &= \|x\|^2 + 2 Re\, \langle x,y\rangle + \|y\|^2\\
&\le \|x\|^2 + 2 |\langle x,y\rangle| + \|y\|^2\\
&\le \|x\|^2 + 2 \|x\|\|y\|+\|y\|^2\\
&= (\|x\| + \|y\|)^2.
\end{align*}

(3)  We compute assuming $\mathbb{H}^N$ is a real Hilbert space:
\begin{align*}
\|x+y\|^2-\|x-y\|^2 &= \|x\|^2 + 2 \langle x,y\rangle + \|y\|^2 
-(\|x\|^2 - 2  \langle x,y\rangle +\|y\|^2)\\
&= 4 \langle x,y\rangle.
\end{align*}
The proof in the complex case is similar.

(4)  Since $\langle x_i,x_j\rangle =0$ for all $i\not= j$, we have
\begin{align*}
\left \|\sum_{i=1}^M x_i\right \|^2 &= \left \langle \sum_{i=1}^M x_i,
\sum_{j=1}^M x_j\right \rangle\\
&= \sum_{i,j=1}^M \langle x_i,x_j\rangle\\
&= \sum_{i=1}^M \langle x_i,x_i\rangle\\
&= \sum_{i=1}^M \|x_i\|^2. \qedhere
\end{align*}
\end{proof}

We now look at subspaces of the Hilbert space.

\begin{definition}
Let $W,V$ be subspaces of $\H^N$.
\begin{enumerate}
\item A vector $x\in \H^N$ is {\bf orthogonal to a subspace\index{orthogonal!to a subspace}}  $W$, denoted $x \perp W$
if
\[ \langle x,y\rangle =0 \mbox{ for all }y\in W.\]
The {\bf orthogonal complement\index{orthogonal!complement}} of $W$ is
\[ W^{\perp}= \{x\in \H:x\perp W\}.\]

\item The subspaces $W,V$ are {\bf orthogonal subspaces\index{orthogonal!subspaces}}, denoted $W\perp V$
if $W \subset V^{\perp}$, that is,
\[ \langle x,y\rangle =0 \mbox{ for all }x\in W,\ \ y\in V.\]
\end{enumerate}
\end{definition}

A simple calculation shows that $W^{\perp}$
is always closed and so if $W$ is closed then $W^{\perp \perp}=W$. Fundamental to Hilbert space theory are {\bf orthogonal projections} as defined next.

\begin{definition}
An operator $P:\H^N \rightarrow \H^N$ is called a {\bf projection\index{projection}} if $P^2=P$.  It is an
{\bf orthogonal projection\index{projection!orthogonal}} if $P$ is also self-adjoint
(See definition \ref{deff}).
\end{definition}

For any subspace $W \subset \H^N$, there is an orthogonal projection of
$\H$ onto $W$ called the {\bf nearest point projection\index{projection!nearest point}}.  One way to define
it is to pick any orthonormal basis $\{e_i\}_{i=1}^K$ for $W$ and define
\[ Px = \sum_{i=1}^K\langle x,e_i\rangle e_i.\]
Note that for all $j=1,2,\ldots,K$ we have
\[ \langle Px,e_j\rangle = \bigg\langle \sum_{i=1}^K\langle x,e_i\rangle e_i,e_j\bigg\rangle
= \sum_{i=1}^K \langle x,e_i\rangle \langle e_i,e_j\rangle = \langle x,e_j\rangle.\]

We need to check that this operator is well defined, that is, we must show it is independent of the choice of basis.

\begin{lemma}
If $\{e_i\}_{i=1}^K$ and $\{g_i\}_{i=1}^K$ are orthonormal bases for $W$, then
\[ \sum_{i=1}^K\langle x,e_i\rangle e_i = \sum_{i=1}^K \langle x,g_i\rangle g_i.\]
\end{lemma}

\begin{proof}
We compute:
\begin{align*}
\sum_{i=1}^K \langle x,e_i\rangle e_i &= \sum_{i=1}^K\left  \langle x, \sum_{j=1}^K
\langle e_i,g_j\rangle g_j\right \rangle e_i\\
&=\left ( \sum_{i,j=1}^K \overline{\langle e_i,g_j\rangle} \langle x,g_j\rangle\right )e_i\\
&= \sum_{i=1}^K \left \langle \sum_{j=1}^K \langle x,g_j\rangle g_j,e_i \right \rangle e_i\\
&= \sum_{j=1}^K\langle x,g_j\rangle g_j.  \qedhere
\end{align*}
\end{proof}

In this case, $Id-P$ is also an orthogonal projection onto $W^{\perp}$.
This projection maps each vector in $\H^N$ onto the unique nearest vector
to $x$ in $W$.  In particular, $\|Px\|\le \|x\|$ for all $x\in \H$.

\begin{proposition}
Let $P$ be an orthogonal projection onto a subspace $W$.  Then
\[ \|x-Px\|\le \|x-y\|, \mbox{ for all } y\in W.\]
\end{proposition}

\begin{proof}
Choose an orthonormal basis $\{e_i\}_{i=1}^K$ for $W$.  Then
\begin{align*}
Re \, \langle x,y \rangle &\leq |\langle x,y\rangle| \\
&= \bigg|\sum_{i=1}^K \langle x,e_i\rangle \langle y,e_i\rangle\bigg|\\
&= \bigg|\sum_{i=1}^K \langle Px,e_i\rangle \langle y,e_i\rangle\bigg|\\
&\le \left ( \sum_{i=1}^K|\langle Px,e_i\rangle|^2\right )^{1/2}\left ( \sum_{i=1}^K|\langle y,e_i\rangle|^2 \right )^{1/2}\\
&\le \|Px\|\|y\|\\
&\le \frac{1}{2}(\|Px\|^2 + \|y\|^2)\\
\end{align*}
Therefore,
\begin{align*}
\|x-y\|^2 &= \|x\|^2 + \|y\|^2 - 2Re\, \langle x,y\rangle\\
&\ge \|x\|^2 + \|y\|^2 - (\|Px\|^2+\|y\|^2)\\
&= \|x\|^2 -\|Px\|^2\\
&= \|x\|^2 +\|Px\|^2 -2\|Px\|^2\\
&= \|x\|^2 + \|Px\|^2 -2\langle Px,Px\rangle\\
&= \|x\|^2 + \|Px\|^2 -2 \langle Px,x\rangle\\
&= \|x-Px\|^2. \qedhere
\end{align*}
\end{proof}

Another important concept is taking orthogonal sums of subspaces of a Hilbert space.

\begin{definition}
If $\{W_i\}_{i\in I}$ (the index set $I$ is allowed to be infinite) are subspaces of a Hilbert space $\H^N$, their {\bf orthogonal direct sum \index{orthogonal!
direct sum}\index{direct sum}} is
\[ \left ( \sum_{i\in I}\oplus W_i\right )_{\ell_2}= \big\{(x_i)_{i\in I}:
x_i \in W_i, \mbox{ and } \sum_{i\in I}\|x_i\|^2
< \infty \big\}\]
and inner product defined by
\[ \big\langle (x_i)_{i\in I},(y_i)_{i\in I}\big\rangle = \sum_{i\in I}\langle x_i,y_i\rangle.\]
It follows that
\[ \big\|(x_i)_{i\in I}\big\|^2= \sum_{i\in I}\|x_i\|^2.\]
\end{definition}

\section{The Basics of Operator Theory}

\begin{definition}
A {\bf linear operator\index{linear operator}\index{operator!linear}} $T:\H^N \rightarrow \H^K$ between Hilbert
spaces $\H^N$ and $\H^K$ satisfies:
\[ T(ax+by) = aT(x) + bT(y)\mbox{ for all }x,y\in \H^N \mbox{ and scalars a,b}.\]
The {\bf operator norm\index{operator norm}\index{norm!operator}\index{operator!norm}}  is
\begin{align*}
\|T\|&= \sup_{\|x\|=1}\|Tx\|\\
&= \sup_{\|x\|\le 1}\|Tx\|\\
&= \sup_{x \not= 0}\frac{\|Tx\|}{\|x\|}.
\end{align*}
\end{definition}

From the definition, for all $x\in \H^N$ we have  $\|Tx\|\le \|T\|\|x\|$. Furthermore, if $T:\H^{N_1} \rightarrow \H^{N_2}$ and $S:\H^{N_2}\rightarrow \H^{N_3}$
then
\[ \|STx\| \le \|S\|\|Tx\|\le \|S\|\|T\|\|x\|\]
showing that
\[ \|ST\|\le \|S\|\|T\|.\]

Linear operators on finite spaces can always be represented as a matrix, so we can work with linear operators and their matrices interchangeably. 

\begin{definition}
Let $T:\H^N \rightarrow \H^K$ be a linear operator, let $\{e_i\}_{i=1}^N$ be
an orthonormal basis for $\H^N$ and let $\{g_j\}_{j=1}^K$ be an orthonormal
basis for $\H^K$.  The {\bf matrix representation\index{matrix representation}} of $T$ (with respect to
these orthonormal bases) is $T=[a_{ij}]_{1\le i \le N,1\le j \le K}$ where
\[ a_{ij}=\langle Te_i,g_j\rangle.\]
\end{definition}

The following definition holds some fundamental concepts that we often work with when dealing with linear operators.

\begin{definition}\label{deff}
Let $T:\H^N \rightarrow \H^K$ be a linear operator.
\begin{enumerate}
\item The {\bf kernel\index{kernel} of T} is
\[ \ker T = \{x\in \H^N:Tx=0\}.\]
The {\bf range \index{range} of T} (or {\bf image\index{image} of T}) is
\[ \ran T = \{Tx:x\in \H^N\}.\]
The {\bf rank\index{rank} of T}, denoted rank T is the dimension of the $\ran T$. A standard result from linear algebra known as
the {\bf \hl{rank-nullity theorem}} states
\[ N=\dim\ker T + \rank T.\]

\item $T$ is {\bf {injective}\index{operator!injective}} if $\ker T = \{0\}$, and {\bf {surjective}\index{operator!surjective}}  if $\ran T = \H^K$.
It is {\bf {bijective}\index{operator!bijective}} if it is both injective and surjective. 

\item The {\bf \hl{adjoint operator}\index{operator!adjoint}} $T^*:\H^K \rightarrow \H^N$ is defined by:
\[ \langle Tx,y \rangle = \langle x,T^*y\rangle \mbox{ for all }x\in \H^N,\ y\in \H^K.\]
Note that $T^{**} =T$ and $(S+T)^* = S^*+T^*$.

\item $T$ is {\bf {bounded}\index{operator!bounded}} if $\|T\|<\infty$.

\item  If $\mathbb{H}^N = \mathbb{H}^K$, then $T$ is {\bf invertible\index{operator!invertible}} if it bounded and there is a bounded linear map $S: \mathbb{H}^N \to \mathbb{H}^N$ so that $TS = ST = Id$. If such an $S$ exists, it is unique and we call it the {\bf \hl{inverse operator}\index{operator!inverse}} of $T$ and denote it by $S = T^{-1}$. In the finite setting, the existence of an inverse is equivalent to $T$ being any of bounded, injective, or surjective.
\end{enumerate}

\end{definition}

A useful alternative method that we will use later for calculating $\|T\|$ is given in the following proposition.

\begin{proposition}\label{lem1}
If $T:\H^N \rightarrow \H^K$ is a linear operator, then
\[ \|T\| = \sup\{ |\langle Tx,y\rangle|:\|x\|=\|y\|=1\}.\]
\end{proposition}

\begin{proof}
For $\|x\|=1=\|y\|$, Cauchy-Schwarz gives
\[ |\langle Tx,y\rangle| \le \|Tx\|\|y\|\le \|T\|\|x\|= \|T\|.\]
Hence,
\[ \|T\| \ge \sup\{ |\langle Tx,y\rangle|:\|x\|=\|y\|=1\}.\]
Conversely,
\begin{align*}
\sup\{|\langle Tx,y\rangle| : \|&x\|=\|y\|=1\} \\
&\geq \sup \bigg \{ \left |\left \langle Tx,\frac{Tx}{\|Tx\|}\right \rangle \right |:\|x\|=1\mbox{ and } Tx \not=0 \bigg\}\\
&= \sup_{\|x\|=1}\frac{\|Tx\|^2}{\|Tx\|}\\
&= \sup_{\|x\|=1}\|Tx\|=\|T\|. \qedhere
\end{align*}

\end{proof}

The following gives a way to identify operators.

\begin{proposition}\label{idfyop}
If $S,T:\H^N \rightarrow \H^K$ are operators satisfying
\[ \langle Tx,y\rangle =0,\mbox{ for all }x,y\in \H^N,\]
then $T=0$.  Hence, if
\[ \langle Tx,y\rangle = \langle Sx,y\rangle\mbox{ for all } x,y\in \H^N,\]
then $S=T$.
\end{proposition}

\begin{proof}
Given $x \in \H^N$, by letting $y=Tx$ we obtain:
\[ 0 =\langle Tx,Tx\rangle = \|Tx\|^2 =0,\]
and so $Tx=0$ and $T=0$. 
\end{proof}

There are important relationships between the kernel of $T$ and the range of $T^*$ and vice versa.

\begin{proposition}\label{lem21}
Let $T:\H^N \rightarrow \H^K$ be a linear operator.  Then
\begin{enumerate}
\item $\ker T = [\ran T^*]^{\perp}$.

\item $[\ker T]^{\perp}= \ran T^*$.

\item  $\ker T^*= [\ran \ T]^{\perp}$.

\item $[\ker T^*]^{\perp}=\ran T$.
\end{enumerate}
\end{proposition}

\begin{proof}
(1) We have $x \in \ker T$ if and only if $Tx=0$ then for all $y\in \H^K$ if and only if $\langle Tx,y\rangle = 0 = \langle x, T^* y\rangle$ for all $y \in H^K$ if and only if $x \in [\ran T^*]^\perp$.

(2)  Observe from (1) that:
\[  [\ker T]^{\perp} = [T^*(\H^K)]^{\perp \perp} = T^*(\H^K) = \ran T^*.\]
The relations (3) and (4) follow by replacing $T$ by $T^*$ in (1) and (2).
\end{proof}

We also have the following relationships for adjoint operators.

\begin{proposition}
For linear operators $T:\H^{N_1} \rightarrow H^{N_2}$ and $S:\H^{N_2} \rightarrow
\H^{N_3}$, we have:
\begin{enumerate}
\item  $(ST)^* = T^*S^*$.

\item  $\|T^*\|=\|T\|$.

\item $\|T^*T\|= \|T\|^2$.
\end{enumerate}
\end{proposition}

\begin{proof}
(1) We compute:
\begin{align*}
\langle x,(ST)^*y\rangle &= \langle STx,y\rangle\\
&= \langle Tx,S^*y\rangle\\
&= \langle x,T^*S^*y\rangle.
\end{align*}

(2)  We use Proposition \ref{lem1} to get:
\begin{align*}
\|T\| &= \sup\{|\langle Tx,y\rangle|:\|x\|=\|y\|=1\}\\
&= \sup\{|\langle x,T^*y\rangle|:\|x\|=\|y\|=1\}\\
&= \|T^*\|.
\end{align*}

(3)  We know that $\|T^*T\|\le \|T^*\|\|T\|= \|T\|^2$.
On the other hand, for all $\|x\|=1$ via Cauchy-Schwarz,
\[ \|Tx\|^2 = \langle Tx,Tx\rangle = \langle T^*Tx,x\rangle
\le \|T^*Tx\|\|x\| \le  \|T^*T\|\|x\|=\|T^*T\|.\]
Hence, $\|T\|^2 \le \|T^*T\|$.
\end{proof}

\begin{definition}
Let $T:\H^N \rightarrow \H^K$ be an injective linear operator.  Then $(T^*T)^{-1}$ exists and the
{\bf \hl{Moore-Penrose inverse} of T}, denoted by $T^{\dagger}$, is defined by
\[ T^{\dagger} := (T^*T)^{-1}T^*.\]
The map $T^{\dagger}$ is a left inverse, that is, $T^{\dagger}T=I$.
\end{definition}

\begin{definition}
A linear operator $T:\H^N \rightarrow \H^K$ is called:
\begin{enumerate}
\item {\bf {self-adjoint}\index{operator!self-adjoint}}, if $\mathbb{H}^N = \mathbb{H}^K$ and  $T=T^*$.

\item {\bf {normal}\index{operator!normal}}, if $\mathbb{H}^N = \mathbb{H}^K$ and  $T^*T=TT^*$.

\item  an {\bf \hl{isometry}}, if $\|Tx\|=\|x\|$ for all $x\in \H^N$.

\item a {\bf \hl{partial isometry}} if $T$ restricted to $[\ker T]^{\perp}$
is an isometry.

\item  {\bf positive \index{operator!positive}}, if $\mathbb{H}^N = \mathbb{H}^K$, $T$ is self-adjoint, and
\[ \langle Tx,x\rangle \ge 0 \mbox{ for all }x\in \H.\]
In this case we write $T\ge 0$.

\item  {\bf unitary\index{operator!unitary}}, if $TT^* = T^*T = Id$.
\end{enumerate}
\end{definition}

It follows that $TT^*$ and $T^* T$ are self-adjoint for any operator $T$. 

\begin{example}\label{frameop}
A fundamental example of a positive, self-adjoint operator important in frame theory is to take 
vectors $\Phi = \{\p_i\}_{i=1}^M$ in $\H^N$ and define the operator:
\[ Sx = \sum_{i=1}^M\langle x,\p_i\rangle \p_i.\]
It follows that 
\[ \langle Sx,x\rangle = \sum_{i=1}^M\langle x,\p_i\rangle \langle \p_i,x\rangle
= \sum_{i=1}^M|\langle x,\p_i\rangle|^2 = \langle x, Sx \rangle\]
showing that it is positive and self-adjoint. This operator is called the {\bfseries frame operator} of the sequence $\Phi$.
\end{example}

Note that given two positive operators $S$ and $T$, the sum $S+T$ is a positive operator but $ST$ may
not be as the next example shows. 

\begin{example}\label{ssnotpos}
Take $S:\mathbb{R}^2 \to \mathbb{R}^2$ to be the operator defined by
\[
S\left[
\begin{array}{c}
x_1 \\ x_2
\end{array}
\right]
=
\left[
\begin{array}{rr}
0 & -1\\
1 & 0
\end{array}
\right]\left[
\begin{array}{c}
x_1 \\ x_2
\end{array}
\right]
=
\left[
\begin{array}{c}
-x_2 \\ x_1
\end{array}
\right]
\]
Then $\langle Sx, x \rangle = 0$ for all $x \in \mathbb{R}^2$ so that $S$ is positive. However, 
\[
S^2 \left[
\begin{array}{c}
x_1 \\ x_2
\end{array}
\right]
=
\left[
\begin{array}{rr}
-1 & 0\\
0 & -1
\end{array}
\right]\left[
\begin{array}{c}
x_1 \\ x_2
\end{array}
\right]
=
\left[
\begin{array}{c}
-x_1 \\ -x_2
\end{array}
\right]
\]
so that $\langle S^2x,x\rangle = -\|x\|^2$ for all $x \in \mathbb{R}^2$ and hence $S^2$ is not positive. 
\end{example}

Nevertheless, we can define inequalities for positive operators:

\begin{definition}
If $S,T$ are positive operators on $\H^N$, we write $S\le T$ if
$T-S\ge 0$.
\end{definition}

\begin{proposition}\label{propp2}
Let $T:\H^N \rightarrow \H^K$ be a linear operator.
\begin{enumerate}
\item  If $\H^N = \H^K$, the following are equivalent:
\begin{enumerate}
\item  $T$ is self-adjoint.

\item $\langle Tx,y\rangle = \langle x,Ty\rangle,\mbox{ for all }x,y \in \H^N.$
\end{enumerate}

\item The following are equivalent:
\begin{enumerate}
\item $T$ is an isometry.

\item  $T^*T=Id$.

\item $ \langle Tx,Ty\rangle = \langle x,y\rangle,\mbox{ for all }x,y\in \H^N.$
\end{enumerate}

\item  The following are equivalent:
\begin{enumerate}
\item $T$ is unitary.

\item  $T$ and $T^*$ are isometries.

\item  $T$ is an isometry and $T^*$ is injective.

\item  $T$ is a surjective isometry.

\item  $T$ is bijective and $T^{-1}=T^*$.
\end{enumerate}

\item If $U$ is any unitary operator, then $\|T\|=\|TU\|=\|UT\|.$
\end{enumerate}

\end{proposition}

\begin{proof}
(1)  Applying Proposition \ref{idfyop}, we have $T^* = T$ if and only if $\langle Tx,y \rangle = \langle T^*x,y \rangle$ for all $x,y$ if and only if $\langle Tx,y \rangle = \langle x, Ty \rangle$ for all $x,y$.

(2)  $(a) \Rightarrow (b)$: We prove the real case, but the complex case is similar using the complex version of the Polarization Identity. We have by the  (real) Polarization Identity that for any $x,y \in \H^N$,
\begin{align*} 
\langle x,y\rangle 
&= \frac{1}{4}\left (\|x+y\|^2 - \|x-y\|^2\right )\\
&= \frac{1}{4}\left ( \|T(x+y)\|^2-\|T(x-y)\|^2
\right )\\
&= \frac{1}{4}\left ( \|Tx\|^2 +\|Ty\|^2+2\langle Tx,Ty
\rangle -[\|Tx\|^2+\|Ty\|^2 - 2 \langle Tx,Ty\rangle]
\right )\\
&= \langle Tx,Ty\rangle \\
&= \langle T^*Tx,y\rangle.
\end{align*}
So, $T^*T=Id$ by Proposition \ref{idfyop}.

$(b)\Rightarrow (c)$:  For any $x,y\in \H^N$ we have
\[ \langle Tx,Ty\rangle = \langle T^*Tx,y\rangle = \langle x,y\rangle.\]

$(c)\Rightarrow (a)$:
(c) implies,
\begin{align*}
\|Tx\|^2 &=  \langle Tx,Tx\rangle \\
&= \langle x,x\rangle \\
&= \|x\|^2.
\end{align*}

(3)  $(a) \Leftrightarrow (b)$:  This is immediate by the definitions
and (2). 

$(b)\Rightarrow (c)$:  By (2, b) $T^*$ is injective.

$(c) \Rightarrow (d)$:  We observe:
\[ \H^K = \{0\}^{\perp} =  [\ker T^*]^{\perp} = T(\H^N) = \ran T.\]

$(d)\Rightarrow (e)$:  Since $T$ is bijective by assumption,
$S=T^{-1}$ exists.  We compute using (2, b):
\[ T^* = T^*I = T^*(TS) = (T^*T)S = IS = S.\]

$(e)\Rightarrow (a)$:  Immediate by the definitions.

(4) Since $U$ is an isometry by (3), the result is now immediate from Lemma \ref{lem1}.
\end{proof}

\begin{definition}
A linear operator $T:\H^N \rightarrow \H^K$ is a {\bf \hl{Hilbert space isomorphism}} if
for all $x,y\in \H^N$ we have
\[ \langle Tx,Ty\rangle = \langle x,y\rangle.\]
Two Hilbert spaces $\mathbb{H}^N$ and $\mathbb{H}^M$ are {\bfseries \hl{isomorphic}} if there is a Hilbert space isomorphism $T:\mathbb{H}^N \to \mathbb{H}^M$.
\end{definition}

Proposition \ref{propp2}(2) implies that $T$ is an isometry and $T^*T=Id$ for any Hilbert space isomorphism $T$. Thus, it is automatically injective. We see in the next proposition that every two Hilbert spaces of the same dimension are isomorphic.

\begin{proposition}
Every two $N$-dimensional Hilbert spaces are Hilbert space isomorphic. Thus, any $N$-dimensional Hilbert space is isomorphic to $\mathbb{C}^N$.
\end{proposition}

\begin{proof}
Let $\{e_i\}_{i=1}^N$ be an orthonormal basis of $\H^N_1$ and let $\{g_i\}_{i=1}^N$ be
an orthonormal basis of $\H^N_2$.  The operator $T$ given by:  $Te_i=g_i$ for
all $i=1,2,\ldots,N$ is clearly a Hilbert space isomorphism.
\end{proof}

In the complex case, an operator $T$ is self-adjoint if and only if $\langle Tx, x \rangle \in \mathbb{R}$ for any $x$. We need a lemma to prove this.

\begin{lemma}\label{tx0}
If $\mathbb{H}^N$ is a complex Hilbert space and $T:\mathbb{H}^N \to \mathbb{H}^N$ is a linear operator satisfying $\langle Tx,x \rangle = 0$ for all $x \in \mathbb{H}^N$, then $T = 0$, the zero operator.
\end{lemma}

\begin{proof}
We prove this by showing that $\langle Tx,y \rangle$ for all $x,y \in \mathbb{H}^N$ and then apply Proposition \ref{idfyop} to conclude that $T = 0$. Note by assumption that we have for all $x,y$, 
\[
0 = \langle T(x + y), x + y \rangle = \langle Tx, y \rangle + \langle y, Tx \rangle
\]
and 
\[
0 = \langle T(x + iy), x + iy \rangle = i \langle Tx, y \rangle - i \langle y, Tx \rangle.
\]
Thus, $\langle Tx, y \rangle = 0$ for all $x,y$ and therefore $T = 0$.
\end{proof}

Now we formally state the theorem and then prove it.

\begin{theorem}\label{tx0thm}
An operator $T:\mathbb{H}^N \to \mathbb{H}^N$ on a complex Hilbert space $\mathbb{H}^N$ is self-adjoint if and only if $\langle Tx, x \rangle \in \mathbb{R}$ for any $x \in \mathbb{H}^N$.  
\end{theorem}

\begin{proof}
We have that $\langle Tx, x \rangle \in \mathbb{R}$ for all $x$ if and only if 
\[
\langle Tx, x \rangle = \overline{\langle Tx, x \rangle} = \langle x, Tx \rangle = \langle T^* x, x \rangle \quad \mbox{for all } x 
\]
if and only if $\langle (T - T^*)x, x \rangle = 0$ for all $x$, which by Lemma \ref{tx0} happens if and only if $T=T^*$.
\end{proof}

\begin{corollary}\label{posgsa}
If $T:\mathbb{H}^N \to \mathbb{H}^N$ is a positive operator on a complex Hilbert space $\mathbb{H}^N$, then $T$ is self-adjoint.
\end{corollary}

\begin{remark}
Notice that Theorem \ref{tx0thm} fails when the Hilbert space is real. This is mainly due to the fact that Lemma \ref{tx0} fails, in particular, the operator given in Example \ref{ssnotpos} is a specific counterexample.
\end{remark} 

Next we introduce the concept of diagonalization, which is fundamental to studying the behavior of operators.

\begin{definition}
Let $T:\H^N \rightarrow \H^N$ be a linear operator.  A nonzero vector
$x\in \H^N$ is an {\bf \hl{eigenvector}} of $T$ with {\bf \hl{eigenvalue}} $\lambda$ if
\[ Tx = \lambda x.\]
The operator $T$ is {\bf \hl{diagonalizable}} if there exists an orthonormal basis
for $\H^N$ consisting of eigenvectors for $T$.
\end{definition}

The above definition is not the {\it standard one}
which states that an operator is diagonalizable if there is
some basis consisting of eigenvectors for $T$.

As an example, if $P$ is a projection from $\H^N$, we have 
\[ \langle Px,x\rangle =\langle P^2x,x\rangle = \langle Px,Px\rangle = \|Px\|^2,\]
so that $P$ is a positive operator. If $P:\H^N \rightarrow W$ and 
dim $W =m$, then $P$ has eigenvalue 1 with multiplicity $m$ and eigenvalue $0$ of multiplicity
$N-m$.

\begin{definition}
Let $T:\H^N \rightarrow \H^N$ be an invertible positive operator with eigenvalues
$\lambda_1 \ge \lambda_2 \ge  \cdots \ge \lambda_N >0$.  The {\bf \hl{condition number}}
is $\frac{\lambda_1}{\lambda_N}$.
\end{definition}

The condition number is an important concept to frame theorists because its gives a way to measure  {\bfseries redundancy}, as we will see later.

The following relationship holds for $T^*T$ and $TT^*$.

\begin{proposition}
If $T:\H^N \rightarrow \H^K$ is any linear operator, then $T^*T$ and 
$TT^*$ have the same nonzero eigenvalues, including multiplicity.
\end{proposition}

\begin{proof}
If $T^{*}Tx = \lambda x$ with $\lambda \neq 0$ and $x \not= 0$, then
$Tx \not=0$ and
\begin{align*}
TT^{*}(Tx) &= T(T^{*}Tx) = T(\lambda x) = \lambda Tx. \qedhere
\end{align*}
\end{proof}

We will see in the next section that $T^*T$ may have zero eigenvalues even if $TT^*$ does not have zero eigenvalues.

Further restrictions on $T$ gives more information on the eigenvalues. The following result is a corollary of Proposition \ref{propp2}.

\begin{corollary}
Let $T:\H^N \rightarrow \H^N$ be a linear operator.

(1)  If $T$ is unitary, then its eigenvalues have modulus one.

(2)  If $T$ is self-adjoint, then its eigenvalues are real.

(3)  If $T$ is positive, then its eigenvalues are nonnegative.
\end{corollary}

Whenever $T$ is diagonalizable, it has a nice representation using its eigenvalues and eigenvectors. It also gives an easy way to compute its norm.

\begin{theorem}
 If $T$ is diagonalizable, so that there exists an orthonormal basis $\{e_i\}_{i=1}^N$ of eigenvectors for
$T$ with respective eigenvalues $\{\lambda_i\}_{i=1}^N$, then 
\[ Tx = \sum_{i=1}^N \lambda_i \langle x,e_i\rangle e_i\mbox{ for all }x\in \H^N,\]
and
\[  \|T\|= \max_{1\le i \le N}|\lambda_i|.\]
\end{theorem}

\begin{proof}
Since $\{e_i\}_{i = 1}^N$ is orthonormal, we can write any $x$ as
\[
x = \sum_{i = 1}^N \langle x,e_i \rangle e_i
\]
and therefore
\begin{align*}
Tx &= T \bigg( \sum_{i = 1}^N \langle x,e_i \rangle e_i \bigg)
= \sum_{i = 1}^N \langle x,e_i \rangle Te_i
= \sum_{i = 1}^N \langle x,e_i \rangle \lambda_i e_i = \sum_{i = 1}^N \lambda_i \langle x,e_i \rangle e_i.
\end{align*}
From this combined with Parseval's identity, we get
\[
\|Tx\|^2 = \sum_{i = 1}^N |\lambda_i|^2 |\langle x, e_i \rangle|^2 \leq \max_{1 \leq i \leq N} |\lambda_i|^2 \sum_{i = 1}^N |\langle x, e_i \rangle|^2 = \max_{1 \leq i \leq N} |\lambda_i|^2 \|x\|^2.
\]
To see that there is an $x$ for which $\|Tx\|$ obtains this maximum, let $j$ be so that $|\lambda_j|$ is the maximum across all $|\lambda_i|$ and then take $x = e_j$ above.
\end{proof}

%



We can classify the diagonalizable operators. In the infinite dimensional
setting this is called the {\bf \hl{Spectral Theorem}}.  We will do a special case which is more instructive
than the general case. Specifically, we will show that all self-adjoint operators are diagonalizable.  We need a series of lemmas.

\begin{lemma}\label{lem5}
If $T$ is a self-adjoint operator on a Hilbert space $\mathbb{H}^N$, then
\[ \|T\|= \sup_{\|x\|=1}|\langle Tx,x\rangle|.\]
\end{lemma}

\begin{proof}
We once again give the proof in the complex setting, but the same general idea works in the real case. To begin, set
\[M:= \sup\limits_{\|x\| = 1}\{|\langle Tx,x\rangle|\}.\] 
We will first show that $M \leq \|T\|$: For any $y \in \H^N$, we have 
\[|\langle Tx, y \rangle | \leq \|Tx\| \|y\| \leq \|T\|\|x\|\|y\|,\]
and thus if $\|x\| = 1$, then $|\langle Tx, x \rangle |\leq \|T\|$. Taking the supremum over all $x \in \H^N$ such that $\|x\|= 1$, we get $M \leq \|T\|$.

Next we will show that $\|T\|\leq M$: For any $x,y \in \mathbb{H}^N$, we have 
\begin{align*}
4\mbox{Re}\,\langle Tx,y\rangle &= \big(\langle Tx,x\rangle + 2 \mbox{Re}\,\langle Tx,y \rangle
+ \langle Ty,y \rangle \big)\\ &\quad - \big(\langle Tx,x \rangle - 2\mbox{Re}\,\langle Tx,y\rangle + \langle Ty,y\rangle \big)\\
 &= \langle T(x+y),x+y\rangle - \langle T(x-y), x-y \rangle
\end{align*}
where the fact that $T$ is self-adjoint was used in the second equality. Hence, 
\begin{align*}
4\mbox{Re}\,\langle Tx,y\rangle &= \| x + y\|^2 \bigg\langle T\left(\dfrac{x + y}{\|x + y\|}\right), \dfrac{x+y}{\|x +y\|} \bigg\rangle \\
&\quad- \| x - y\|^2 \bigg\langle T\left(\dfrac{x - y}{\|x - y\|}\right), \dfrac{x-y}{\|x -y\|} \bigg\rangle\\
&\leq M(\|x+y\|^2 + \|x-y\|^2)=2M(\|x\|^2+\|y\|^2).
\end{align*}
Note that there exists a $\theta \in \left[0,2\pi\right)$ such that $e^{i\theta}\langle Tx,y\rangle = |\langle Tx,y\rangle|$. Now, replace $y$ with $e^{-i\theta}y$ to obtain:
\[4|\langle Tx, y\rangle|\leq 2M(\|x\|^2+\|y\|^2).\]
Finally, if $\|x\| = 1$ and $y =Tx/\|Tx\|$, then
\[\|Tx\|=\bigg\langle Tx, \frac{Tx}{\|Tx\|}\bigg\rangle \leq M.\]
Combining both steps gives $\|T\|=M$ as desired. 
\end{proof}

\begin{lemma}\label{Nor}
If $T$ is normal then $\|Tx\|=\|T^*x\|$ for all $x\in \H^N$.
\end{lemma}

\begin{proof}
We compute:
\begin{align*}
\|Tx\|^2 &= \langle Tx,Tx\rangle\\
&= \langle T^*Tx,x\rangle\\
&= \langle TT^*x,x\rangle\\
&= \langle T^*x,T^*x\rangle\\
&= \|T^*x\|^2. \qedhere
\end{align*}
\end{proof}

\begin{lemma}
If $T$ is normal and
\[ Tx = \lambda x\]
 for some $x\in \H^N$ and scalar $\lambda$, then
\[ T^*x = \overline{\lambda}x.\]
\end{lemma}

\begin{proof}
If $T$ is normal, then $T-\lambda \cdot Id$ is normal.  Applying Lemma \ref{Nor}  gives
\begin{align*}
Tx=\lambda x &\Leftrightarrow \|(T-\lambda\cdot Id)x\|=0\\
&\Leftrightarrow \|(T-\lambda \cdot Id)^*x\|=0\\
&\Leftrightarrow \|T^*x-\overline{\lambda}x\|=0\\
&\Leftrightarrow T^*x = \overline{\lambda}x. \qedhere
\end{align*}
\end{proof}

Finally, we will need:

\begin{lemma}\label{lem6}
If $T$ is normal and $Tx = \lambda x$, then
\[ T\left ( [x]^{\perp} \right ) \subset [x]^{\perp}.\]
\end{lemma}

\begin{proof}
We have $y \perp x$ if and only if
\[ \langle Ty,x\rangle = \langle y,T^*x\rangle= \langle y,\overline\lambda x\rangle =0\]
if and only if $Ty \perp x$.
\end{proof}

Now we are ready to prove the main theorem on diagonalization.

\begin{theorem}\label{spec}
If $T$  is a self-adjoint,
then $T$ is diagonalizable.
\end{theorem}

\begin{proof}By Lemma \ref{lem5}, we can choose
$x\in \H^N$ with $\|x\|=1$ and $\|T\|=|\langle Tx,x\rangle|$. By Theorem \ref{tx0thm} we know that $\langle Tx, x \rangle \in \mathbb{R}$, even if the Hilbert space is complex. Therefore, $\|T\| = \langle Tx, x \rangle$ or $\|T\| = - \langle Tx,x \rangle$. If $\|T\| = \langle Tx,x \rangle$, then
\begin{align*}
\|(T-\|T\|\cdot Id)x\|^2 &=
\langle Tx-\|T\|x,Tx - \|T\|x\rangle\\
&= \langle Tx,Tx\rangle + \|T\|^2 -2 \|T\|\langle Tx,x\rangle\\
&\leq2\|T\|(\|T\|-\langle Tx,x\rangle)\\
&=0
\end{align*}
and so $\lambda_1=\|T\|$ is an eigenvalue for $T$ with eigenvector $e_1=x$. On the other hand, if $\|T\| = -\langle Tx, x \rangle$, then $\|(T + \|T\|\cdot Id) x \|= 0$ by a similar argument, so that $\lambda_1 = -\|T\|$.

Now, if we
restrict the operator $T$ to $[e_1]^{\perp}$, we have by Lemma \ref{lem6} that
\[ T|_{[e_1]^{\perp}}:[e_1]^{\perp}\rightarrow [e_1]^{\perp},\]
and is still self-adjoint.  So we can repeat the above argument to
get a second orthogonal eigenvector of norm one and eigenvalue of the form:
\[ \lambda_2 = \big\|T|_{[e_1]^{\perp}}\big\| \quad \mbox{ or } \quad \lambda_2 = -\big\|T|_{[e_1]^{\perp}}\big\|,\ \ e_2\in [e_1]^{\perp}.\]
Continuing, we get a sequence of eigenvalues $\{\lambda_i\}_{i=1}^N$ so that
\[ \lambda_j = \big\|T|_{[\spn\, \{e_i\}_{i=1}^{j-1}]^{\perp}}\big\| \quad \mbox{ or } \quad \lambda_j = -\big\|T|_{[\spn\, \{e_i\}_{i=1}^{j-1}]^{\perp}}\big\|\]
with corresponding orthonormal eigenvectors $\{e_i\}_{i = 1}^N$. 
\end{proof}

It is possible to take well-defined {\bf \hl{powers of an operator}\index{operator!powers of}} when it is positive, invertible, and diagonalizable.

\begin{corollary}
Let $T$ be an positive, invertible, diagonalizable operator on $\H^N$ with eigenvectors 
$\{e_i\}_{i=1}^N$ and respective eigenvalues $\{\lambda_i\}_{i=1}^N$.
For any nonnegative $a\in \R$ we define the operator $T^a$ by
\[ T^ax = \sum_{i=1}^N \lambda_i^a\langle x,e_i\rangle e_i,\mbox{ for all }
x\in \H^N.\]
\end{corollary}

Thus, $T^a$ is also positive operator and $T^aT^b = T^{a+b}$ for all $a,b \in \mathbb{R}$. In particular, $T^{-1}$ and $T^{-1/2}$ are positive operators. We also note that the definition makes sense for nonnegative powers when $T$ is not invertible.

Notice that Corollary \ref{posgsa} combined with Theorem \ref{spec} gives that all positive operators on a complex Hilbert space are diagonalizable. However, it is possible for a positive operator on a real Hilbert space to not be diagonalizable and hence we cannot obtain powers as defined above. One again, the operator $S$ as given in Example \ref{ssnotpos} provides a counterexample since this operator has no eigenvalues over the reals.

We conclude this section with the concept of trace. In order for it to be well-defined, we need the following proposition.

\begin{proposition}
Let $T$ be a linear operator on $\H^N$ and let $\{e_i\}_{i=1}^N$ and $\{g_i\}_{i=1}^N$
be orthonormal bases for $\H^N$. Then
\[ \sum_{i=1}^N\langle Te_i,e_i\rangle = \sum_{i=1}^N \langle Tg_i,g_i\rangle.\]
\end{proposition}

\begin{proof}
We compute:
\begin{align*}
\sum_{i=1}^N\langle Te_i,e_i\rangle&= \sum_{i=1}^N \left \langle
\sum_{j=1}^N \langle Te_i,g_j\rangle g_j,e_i\right \rangle\\
&= \sum_{i=1}^N\sum_{j=1}^N\langle Te_i,g_j\rangle \langle g_j,e_i\rangle\\
&=\sum_{i=1}^N\sum_{j=1}^N\langle e_i,T^*g_j\rangle \langle g_je_i\rangle\\
&= \sum_{j=1}^N \left \langle \sum_{i=1}^N \langle g_j,e_i\rangle e_i,T^*g_j\right \rangle\\
&= \sum_{j=1}^N \langle g_j,T^*g_j\rangle\\
&= \sum_{j=1}^N\langle Tg_j,g_j\rangle. \qedhere
\end{align*}
\end{proof}

\begin{definition}
Let $T:\H^N \rightarrow \H^N$ be an operator.  The {\bf \hl{trace}\index{operator!trace} of T} is
\[ Tr\ T = \sum_{i=1}^N \langle Te_i,e_i\rangle,\]
where $\{e_i\}_{i=1}^N$ is any orthonormal basis of $\H^N$. The previous proposition shows that this quantity is independent of the choice of orthonormal basis and therefore the trace is well-defined.
\end{definition}

When $T$ is diagonalizable, the trace can be easily computed using its eigenvalues.

\begin{corollary} \label{cor20}
If $T:\mathbb{H}^N \to \mathbb{H}^N$ is a diagonalizable operator  with eigenvalues
$\{\lambda_i\}_{i=1}^N$, then
\[ Tr\ T = \sum_{i=1}^N \lambda_i.\]
\end{corollary}

\begin{proof}
If $\{e_i\}_{i=1}^N$ is an orthonormal basis of eigenvectors associated to $\{\lambda_i\}_{i =1}^N$, then
\begin{align*}Tr\ T = \sum_{i=1}^N\langle Te_i,e_i\rangle = \sum_{i=1}^M\langle \lambda_i e_i,
e_i\rangle &= \sum_{i=1}^N \lambda_i. \qedhere\end{align*}
\end{proof}

\section{Hilbert Space Frames}\label{Intro}
\setcounter{equation}{0}

Hilbert space frames were introduced by Duffin and Schaeffer in 
1952 \cite{DS} to address some deep questions in non-harmonic Fourier
series. 
The idea was to weaken Parseval's Identity as given in Proposition \ref{parsevalid}.
We do not need to have an orthonormal sequence to have equality
in Parseval's identity.  For example, if $\{e_i\}_{i=1}^N$ and
$\{g_i\}_{i=1}^N$ are orthonormal bases for a Hilbert space
$\H^N$ then 
\[
\bigg\{\frac{1}{\sqrt{2}}e_1,\frac{1}{\sqrt{2}}g_1,\frac{1}{\sqrt{2}}e_2,\frac{1}{\sqrt{2}}g_2, \ldots, \frac{1}{\sqrt{2}}e_N,\frac{1}{\sqrt{2}}g_N\bigg\}_{i=1}^N
\]
satisfies Parseval's identity.


\begin{definition}
A family of vectors $\{\p_i\}_{i=1}^M$ 
is a {\bf \hl{frame}} for a Hilbert space $\H^N$ if
there are constants $0<A\le B <\infty$ so that for all $x\in \H^N$
\begin{align}\label{E1f}
 A \|x\|^2 \le \sum_{i=1}^M |\langle x,\p_i\rangle|^2
\le B \|x\|^2.
\end{align}
We include some common, often used terms:
\begin{itemize}
\item The constants $A$ and $B$ are called {\bf lower and upper} {\bf {frame bounds}\index{frame!bounds}}, respectively,
for the frame. The largest lower frame bound and the smallest upper frame bound are called the {\bfseries {optimal frame bounds}\index{frame!optimal frame bounds}}.

\item If $A=B$ this is an {\bf $A$-{tight frame}\index{frame!tight}} and if $A=B=1$ this is a 
{\bf {Parseval frame}\index{frame!Parseval}}. 

\item If $\|\p_i\| = \|\p_j\|$ for
all $i,j\in I$, this is an {\bf {equal norm frame}\index{frame!equal norm}} and if
$\|\p_i\|=1$ for all $i\in I$
this is a {\bf {unit norm frame}\index{frame!unit norm}}.

\item If $\|\phi_i\|=1$ for $i\in I$ and there exists a constant $c$ so that $|\langle \p_i,\p_j\rangle| = c$ for all $i \neq j$, then the frame is called an {\bfseries {equiangular frame}\index{frame!equiangular}}.

\item The values $\{\langle x, \p_i \rangle\}_{i = 1}^M$ are called the {\bfseries {frame coefficients}\index{frame!coefficients}} of the vector $x$ with respect to the frame.

\item The frame is a {\bf {bounded frame}\index{frame!bounded}} if $\min_{1\le i \le M}\|\p_i\| >0$. 

\item If only the right hand inequality holds in (\ref{E1f}) we call $\{\p_i\}_{i\in I}$ a $B${\bf -\hl{Bessel}} sequence or simply {\bf {Bessel}} if explicit reference to the constant is not needed.  
\end{itemize}
\end{definition}

It follows
from the left-hand-side of inequality
(\ref{E1f}), that the closed linear span of
a frame must equal the Hilbert space and so $M \geq N$.  In the finite dimensional case, spanning
is equivalent to being a frame.

\begin{proposition}\label{lem102}
$\Phi = \{\p_i\}_{i=1}^M$ is a frame for $\H^N$ if and only if $\spn\,\Phi = \H^N$.
\end{proposition}

\begin{proof}We only need to prove the {\it if} part.
For the right hand inequality,
\[ \sum_{i=1}^M |\langle x,\p_i\rangle|^2 \le \sum_{i=1}^M\|x\|^2\|\p_i\|^2 \le B \|x\|^2,\]
where 
\[ B = \sum_{i=1}^M \|\p_i\|^2.\]

For the left hand inequality, we proceed by contradiction. Suppose we can find a sequence $\{x_n\}_{n=1}^\infty$ with $\|x_n\|=1$ (by scaling) so that
\[ \sum_{i=1}^M|\langle x_n,\p_i\rangle|^2 \le \frac{1}{n},\]
and thus there is a norm convergent subsequence $\{x_{n_j}\}_{j=1}^{\infty}$ of $\{x_n\}_{n=1}^{\infty}$,
say $x_{n_j}\rightarrow x$.  Then
\[ \sum_{i=1}^M |\langle x,\p_i\rangle|^2 = \lim_{j\rightarrow \infty}\sum_{i=1}^M\big|\langle x_{n_j},\p_i\rangle\big|^2
=0.\]
That is, $x \perp \p_i$ for all $i=1,2,\ldots,M$ and so $\Phi$ does not span $\H^N$.
\end{proof}

Spanning does not necessarily imply that a sequence is a frame when the space is infinite dimensional. For example, suppose $\{e_i\}_{i = 1}^{\infty}$ is an orthonormal basis for an infinite dimensional Hilbert space $\mathbb{H}$. Then the sequence $\{ e_i/i\}_{i = 1}^\infty$ spans the space, but is not frame since a lower frame bound does not exist.

It is important to note that there are no restrictions put on the frame vectors.  For example,
if $\{e_i\}_{i=1}^{N}$ is an orthonormal basis for  
$\H^N$, then $$\{e_1,0,e_2,0,e_3,0,\ldots,e_N,0\}$$ 
and
\[ \left \{e_1,\frac{e_2}{\sqrt{2}},\frac{e_2}{\sqrt{2}},\frac{e_3}{\sqrt{3}},
\frac{e_3}{\sqrt{3}},\frac{e_3}{\sqrt{3}},\ldots,\frac{e_N}{\sqrt{N}},\cdots,\frac{e_N}{\sqrt{N}} \right \}\]
are both Parseval frames for $\H^N$. That is, zeros and repetitions are allowed. 

The smallest redundant family in $\R^2$ has three vectors and can be chosen
to be a unit norm, tight, and equiangular frame called the {\bf\hl{Mercedes Benz Frame}\index{frame!Mercedes Benz}}, given by
\[
\left\{\sqrt{\frac{2}{3}} \left(\begin{array}{r} 0 \\ 1 \end{array}\right),
\sqrt{\frac{2}{3}} \left(\begin{array}{r} \frac{\sqrt{3}}{2} \\ -\frac{1}{2} \end{array}\right),
\sqrt{\frac{2}{3}} \left(\begin{array}{r} -\frac{\sqrt{3}}{2} \\ -\frac{1}{2} \end{array}\right)
\right\}.
\]
Drawing the vectors might illuminate where it got its name.

\subsection{Frame Operators}

If $\{\p_i\}_{i=1}^M$ is a frame for $\H^N$ with frame bounds $A,B$, define
the {\bf \hl{analysis operator}\index{operator!analysis}} of the frame $T:\H^N \rightarrow \ell_2^M$ to be
\[ 
Tx = \sum_{i=1}^M\langle x,\p_i\rangle e_i = \big\{\langle x,\p_i\rangle\big\}_{i = 1}^M,\ \ \mbox{for all $x\in \H^N$},
\] 
where $\{e_i\}_{i=1}^M$ is the natural orthonormal basis of $\ell_2^M$.  It follows that
\[ \|Tx\|^2 = \sum_{i=1}^M|\langle x,\p_i\rangle|^2,\]
so $\|T\|^2$ is the optimal Bessel bound of the frame.

The
adjoint of the analysis operator is the {\bf \hl{synthesis operator}\index{operator!synthesis}} which
is given by
\[ T^{*}e_i = \p_i.\]
Note that the matrix representation of the synthesis operator 
of a frame $\{\p_i\}_{i=1}^M$ is the
$N \times M$ matrix with the frame vectors as its columns.
\[ T^*=
\left[\begin{array}{cccc}
| & | & \ldots & | \\
\p_1 & \p_2 & \ldots & \p_M\\
| & | &  \ldots & |
\end{array} \right]
\]
In practice, we often work with the matrix representation of the synthesis operator with respect to the eigenbasis of the frame operator (as defined below). It will be shown later that the rows and columns of this matrix representation must satisfy very specific properties that proves useful in constructing frames. See Proposition \ref{presyn} and Proposition \ref{synprop}.

\begin{theorem}\label{classthm}
Let $\{\p_i\}_{i=1}^M$ be a family of vectors in a Hilbert space $\H^N$.
The following are equivalent:
\begin{enumerate}
\item $\{\p_i\}_{i=1}^M$ is a frame for $\H$.

\item The operator $T^{*}$ is bounded, linear, and surjective.

\item  The operator $T$ bounded, linear, and injective.
\end{enumerate}
Moreover, $\{\p_i\}_{i=1}^M$ is a Parseval frame if and only if
the synthesis operator is a quotient map (that is, a partial isometry) if and
only if $T^* T=Id$ if and only if $T$ is an isometry.
\end{theorem}

\begin{proof}
$(1) \Leftrightarrow (2)$ is immediate by Proposition \ref{lem102}, and $(2)\Leftrightarrow (3)$ is immediate by Proposition \ref{lem21}. Proposition \ref{propp2} gives the moreover part.
\end{proof}

The {\bf \hl{frame operator}\index{frame!operator}} for the frame is $S = T^{*}T:\H^N \rightarrow
\H^N$ given by
\[ Sx = T^{*}Tx = T^{*} \left ( \sum_{i=1}^M\langle x,\p_i\rangle e_i \right )
= \sum_{i=1}^M\langle x,\p_i\rangle T^{*}e_i = \sum_{i=1}^M\langle x,\p_i
\rangle \p_i.\]
A direct calculation, as given in Example \ref{frameop}, yields
\begin{align}\label{sxe}\langle Sx,x\rangle = \sum_{i =1}^M|\langle x,\p_i\rangle|^2.\end{align}

\begin{proposition}
The frame operator of a frame is a {\bf positive, self-adjoint, and invertible
operator} on $\H^N$.  Moreover, if $A$ and $B$ are frame bounds, then $S$ satisfies the {\bf \hl{operator inequality}\index{operator!inequality}}
\[ A\cdot Id \le S \le B \cdot Id.\]
\end{proposition}

\begin{proof}
Example \ref{frameop} shows that it is positive and self-adjoint. We check the operator inequality:
\begin{align*}\langle Ax,x\rangle &= A\|x\|^2 \le \sum_{i=1}^M|\langle x,\p_i\rangle|^2 = \langle Sx,x\rangle
\le B \|x\|^2 = \langle Bx,x\rangle 
\end{align*}
for all $x \in \mathbb{H}^N$. Note that this also shows that $S$ is invertible.
\end{proof}

The frame operator can be used to {\bf reconstruct} vectors in the space using the computation
\begin{align*}
x &= S^{-1}Sx\\
&= SS^{-1}x\\
&= \sum_{i=1}^M\langle S^{-1}x,\p_i\rangle \p_i\\
&= \sum_{i=1}^M\langle x,S^{-1}\p_i\rangle \p_i
\end{align*}
Also,
\begin{align*}
\sum_{i=1}^M\langle x,S^{-1/2}\phi_i\rangle S^{-1/2}
\phi_i
&= S^{-1/2}\left ( \sum_{i=1}^M\langle S^{-1/2}x,\phi_i\rangle \phi_i \right )\\
&= S^{-1/2}(S(S^{-1/2}x)\\
&= x.
\end{align*}
It follows that $\{S^{-1/2}\phi_i\}_{i=1}^M$ is a Parseval
frame.
Since $S$ is invertible, the family
$\{S^{-1}\p_i\}_{i=1}^M$ is also a frame for $\H^N$ called the
{\bf \hl{canonical dual frame}\index{frame!canonical dual}\index{dual frame!canonical}}. See Subsection \ref{duals} for the definition and properties of dual frames.

The following result is basically just a restatement of the definitions and known facts.

\begin{proposition}
Let $\Phi=\{\p_i\}_{i=1}^M$ be a frame for $\H^N$ with analysis operator $T$ and frame
operator $S$.  The following are equivalent:
\begin{enumerate}
\item $\{\p_i\}_{i=1}^M$ is an $A$-tight frame for $\H^N$.

\item $S=A\cdot Id$.

\item  For every $x\in \H^N$,
\[ x = \dfrac{1}{A}\sum_{i=1}^M\langle x,\p_i\rangle\p_i.\]

\item  For every $x\in \H^N$,
\[ A\|x\|^2 = \sum_{i=1}^M|\langle x,\p_i\rangle|^2.\]

\item  $T/\sqrt{A}$ is an isometry.
\end{enumerate}
Moreover, if $\Phi$ is a Parseval frame then $A=1$ above.
\end{proposition}

Tight frames have the property that all of the eigenvalues of the frame operator coincide and are equal the frame bound. For an arbitrary frame, it turns out the smallest and largest eigenvalues of the frame operator are the optimal lower and upper frame bounds, respectively.

\begin{theorem}
Suppose $\{\p_i\}_{i=1}^M$ is a frame for $\mathbb{H}^N$ with frame operator $S$ with eigenvalues $\lambda_1 \geq \lambda_2 \geq \cdots \geq \lambda_N$. Then $\lambda_1$ is the optimal upper frame bound and $\lambda_N$ is the optimal lower frame bound.
\end{theorem}

\begin{proof}
Let $\{e_j\}_{j=1}^N$ be an orthonormal eigenbasis of the frame operator $S$ with associated eigenvalues $\{\lambda_j\}_{j=1}^N$ given in decreasing order. Now, given an $x \in \mathbb{H}^N$, write
\[
x = \sum_{j=1}^N \langle x, e_j \rangle e_j
\]
to obtain
\[
Sx = \sum_{j=1}^N \lambda_j \langle x, e_j \rangle e_j.
\]
By Equation (\ref{sxe}), this gives that
\begin{align*}
\sum_{i=1}^M |\langle x, \p_i\rangle|^2 &= \langle Sx, x\rangle = \sum_{i,j=1}^N \big\langle \lambda_i \langle x, e_i\rangle e_i, \langle x, e_j\rangle e_j \big \rangle\\
&= \sum_{j=1}^N \lambda_j | \langle x, e_j \rangle |^2 \leq \lambda_1 \|x\|^2
\end{align*}
proving that $\lambda_1$ is an upper frame bound. To see that it is optimal, note 
\begin{align*}
\sum_{i=1}^M |\langle e_1, \p_i\rangle|^2 &= \langle Se_1, e_1\rangle = \lambda_1.
\end{align*}
The lower bound is proven similarly.
\end{proof}

Another type of sequence that we often deal with are \emph{Riesz bases}, which rids of the orthonormality assumption, but retains a unique composition.

\begin{definition}
A family of vectors $\{\p_i\}_{i=1}^N$ in a Hilbert space $\H^N$
is a {\bf \hl{Riesz basis}} if there are constants $A,B>0$ so
that for all families of scalars $\{a_i\}_{i=1}^N$ we have
\[ A \sum_{i=1}^N|a_i|^2 \le \bigg\|\sum_{i=1}^N a_i\p_i\bigg\|^2 \le
B \sum_{i=1}^N|a_i|^2.\]
\end{definition}

The following gives an equivalent formulation, in particular, Riesz bases are precisely sequences of vectors that are images of orthonormal bases under an invertible map. It follows directly from the definitions.

\begin{proposition}
  Let $\Phi=\{\p_i\}_{i = 1}^N$ be a family vectors in $\mathbb{H}^N$. Then the following are equivalent.
\begin{enumerate}
\item $\Phi$ is a Riesz basis for $\mathbb{H}^N$ with Riesz bounds $A$ and $B$.
\item For any orthonormal basis $\{e_i\}_{i =1}^N$ for $\mathbb{H}^N$, the operator $F$ on $\mathbb{H}^N$ given by $Fe_i = \p_i$ for all $i = 1,2,\ldots,N$ is an invertible operator with $\|F\|^2 \leq B$ and $\|F^{-1}\|^{-2} \geq A$.
\end{enumerate}
It follows that if $\Phi$ is a Riesz basis with bounds $A$ and $B$, then it is a frame with these same bounds.
\end{proposition}

Next we see that applying an invertible operator to a frame still gives a frame.

\begin{proposition}\label{1.9}
Let $\Phi=\{\p_i\}_{i=1}^M$ be a sequence of vectors in $\H^N$ with analysis operator
$T$ and let $F$ be a linear operator on $\H^N$.  Then the analysis operator of the sequence
$F\Phi=\{F\p_i\}_{i=1}^M$ is given by
\[T_{F\Phi}=TF^*.\]
Moreover, if $\Phi$ is a frame for $\mathbb{H}^N$ and $F$ is invertible, then $F\Phi$ is a also a frame for $\mathbb{H}^N$.
\end{proposition}

\begin{proof}
For $x\in \H^N$,
\[ T_{F\Phi}x=\big \{ \langle x,F\p_i\rangle\big \}_{i=1}^M = \big \{ \langle F^*x,\p_i\rangle \big\}_{i=1}^M = TF^*x.\]
The moreover part follows from Theorem \ref{classthm}.
\end{proof}

Furthermore, if we apply an invertible operator to a frame,
there is a specific form that
the frame operator of the new frame must take.

\begin{proposition}\label{propp1}
Suppose $\Phi = \{\p_i\}_{i=1}^M$ is a frame for $\H^N$ 
with frame operator $S$ and $F$
is an invertible operator on $\H^N$.  Then the frame
operator of the frame $F\Phi = \{F\p_i\}_{i=1}^M$ is the operator $FSF^*$.
\end{proposition}
\begin{proof}
This follows immediately from Proposition \ref{1.9} and the definition.
\end{proof}

Corollary \ref{cor20} concerning the trace formula for diagonalizable operators,
has a corresponding result for Parseval frames.

\begin{proposition}[\hl{Trace formula for Parseval frames}]
Let $\Phi = \{\p_i\}_{i=1}^M$ be a parseval frame on $\H^N$ and let $F$ be a linear
operator on $\H^N$.  Then
\[ Tr(F) = \sum_{i=1}^M \langle F\p_i,\p_i\rangle.\]
\end{proposition}

\begin{proof}
If $\{e_i\}_{i=1}^N$ is an orthonormal basis for $\H^N$ then by definition
\[ Tr(F)= \sum_{j=1}^N \langle Fe_i,e_i\rangle.\]
This along with the fact that $\Phi$ is Parseval gives that
\begin{align*}
Tr(F) &= \sum_{j=1}^N\left \langle \sum_{i=1}^M  \langle Fe_j, \p_i\rangle \p_i,e_j\right \rangle\\
&=\sum_{j=1}^N\sum_{i=1}^M\langle e_j,F^*\p_i\rangle \langle \p_i,e_j\rangle\\
&=\sum_{i=1}^M \left \langle \sum_{j=1}^N\langle \p_i,e_j\rangle e_j,F^*\p_i\right \rangle\\
&= \sum_{i=1}^M \langle \p_i,F^*\p_i\rangle\\
&= \sum_{i=1}^M \langle F\p_i,\p_i\rangle. \qedhere
\end{align*}
\end{proof}

\begin{definition}
 Two frames
$\{\p_i\}_{i=1}^M$ and $\{\psi_i\}_{i=1}^M$ in a Hilbert space
$\H^N$ are {\bf \hl{isomorphic frames}\index{frame!isomorphic}} if
there exists a bounded, invertible operator $L: \mathbb{H}^N \to \mathbb{H}^N$ so that $L\p_i=\psi_i$ for all $1\leq i \leq M$.
We say they are {\bf \hl{unitarily isomorphic frames}\index{isomorphic frames!unitarily}\index{frame!unitarily isomorphic}} if $L$ is a unitary operator.
\end{definition}

\begin{proposition}
Let $\Phi=\{\p_i\}_{i=1}^M$ and $\Psi=\{\psi_i\}_{i=1}^M$ be frames for $\H^N$ with analysis operators
$T_1$ and $T_2$ respectively.  The following are equivalent:
\begin{enumerate}
\item $\Phi$ and $\Psi$ are isomorphic.

\item  $\ran T_1=\ran T_2$.

\item $\ker T_1^*= \ker T_2^*$.
\end{enumerate}
Moreover, in this case $F=T_2^*(T_1^*|_{\ran T_1})^{-1}$ satisfies $F\p_i = \psi_i$ for all $i=1,2,\ldots,M.$
\end{proposition}

\begin{proof}
The equivalence of (2) and (3) follows by Proposition \ref{lem21}.

$(1) \Rightarrow (3)$:  Let $F\p_i=\psi_i$ be a well-defined invertible operator on $\H^N$.  Then 
by Proposition \ref{1.9} we have that $T_2=T_1F^*$ and hence $FT_1^*=T_2^*$.  Since
$F$ is invertible, (3) follows.

$(2)\Rightarrow (1)$:  Let $P$ be the orthogonal projection onto $W=\ran T_1 = \ran T_2$.  Then $(Id-P)$ is an orthogonal projection onto $W^\perp = \ker T_1^* = \ker T_2^*$ so that 
$$\p_i = T_1^*e_i = T_1^*Pe_i + T_1^*(Id-P)e_i = T_1^*Pe_i$$ 
and similarly $\psi_i = T_2^*Pe_i$.  The operators $T_1^*$ and $T_2^*$ both map $W$ bijectively onto $\H^N$.  Therefore, the operator $F:=T_2^*(T_1^*|_{W})^{-1}$ maps $\H^N$ bijectively onto itself.  Consequently, for all $i=1,2,\ldots,M$ we have
\begin{align*}
 F\p_i &= T_2^*(T_1^*|_{W})^{-1}T_1^*Pe_i = T_2^*Pe_i=\psi_i. \qedhere
 \end{align*}
\end{proof}

As a consequence of Proposition \ref{propp1} we have:

\begin{theorem}\label{Tight}
Every frame $\{\p_i\}_{i=1}^M$ (with frame operator $S$) is isomorphic to
the Parseval frame $\{S^{-1/2}\p_i\}_{i\in I}$.
\end{theorem}

\begin{proof}
The frame operator of $\{S^{-1/2}\p_i\}_{i=1}^N$ is $S^{-1/2}S(S^{-1/2})^*=Id$.
\end{proof}

As a consequence, only unitary operators can map Parseval frames to Parseval
frames.

\begin{corollary}
If two Parseval frames $\Phi=\{\p_i\}_{i=1}^M$ and $\Psi=\{\psi_i\}_{i=1}^M$ are isomorphic, then they are unitarily isomorphic.
\end{corollary}

\begin{proof}
Since both frames have the identity as their frame operator, if $F$ maps one Parseval
frame to another and is invertible, then by Propositon \ref{propp1},
\[ Id = F(Id)F^* = FF^*.\]
Since $F^*F$ is injective, $F$ is a unitary operator by Proposition \ref{propp2}.
\end{proof}

We can always ``move'' one frame operator to another.

\begin{proposition}
Let $\Phi=\{\p_i\}_{i=1}^M$ and $\Psi=\{\psi_i\}_{i=1}^M$ be frames for $\H^N$ with frame
operators $S_1$ and $S_2$ respectively.  Then there exists an invertible operator
$F$ on $\H^N$ so that $S_1$ is the frame operator of the frame $F\Psi = \{F\psi_i\}_{i=1}^M$.
\end{proposition}

\begin{proof}
If $S$ is the frame operator of $F\Psi$, letting $F=S_1^{1/2}S_2^{-1/2}$ we have
\begin{align*}
S = FS_2F^* &= (S_1^{1/2}S_2^{-1/2})S_2(S_1^{1/2}S_2^{-1/2})^* = S_1.\qedhere
 \end{align*}
\end{proof}

\subsection{Dual Frames}\label{duals}

We begin with the definition.

\begin{definition}
If $\Phi=\{\p_i\}_{i=1}^M$ is a frame for $\H^N$, a frame $\{\psi_i\}_{i=1}^M$ for
$\H^N$ is called a {\bf \hl{dual frame}\index{frame!dual}} for $\Phi$ if
\[ \sum_{i=1}^M\langle x,\psi_i\rangle \p_i=x,\mbox{ for all }x\in \H^N.\]
\end{definition}

It follows that the \hl{canonical dual frame}\index{dual frame!canonical}\index{frame!canonical dual} $\{S^{-1}\p_i\}_{i=1}^M$, where $S$ is
the frame operator of $\Phi$, is a dual frame.  But there are many other dual
frames in general.  

\begin{proposition}\label{PAP1}
Let $\Phi=\{\p_i\}_{i=1}^M $ and $\Psi=\{\psi_i\}_{i=1}^M$ be frames for $\H^N$ with analysis operators
$T_1$, $T_2$ respectively.  The following are equivalent:
\begin{enumerate}
\item  $\Psi$ is a dual frame of $\Phi$.

\item $T_1^*T_2=Id$.
\end{enumerate}
\end{proposition}

\begin{proof}
Note that for any $x \in \mathbb{H}^N$,
\[ T_1^*T_2x = T_1^*\left [ (\langle x,\psi_i\rangle)_{i=1}^M\right ]
= \sum_{i=1}^M \langle x,\psi_i\rangle \p_i.\]
The result is now immediate.
\end{proof} 

\begin{theorem}
Suppose $\Phi=\{\p_i\}_{i=1}^M$ is a frame with analysis operator $T_1$ and frame operator $S$. The class of all dual frames of $\Phi$  are frames
of the form  $\{\eta_i\}_{i=1}^M:=\{S^{-1}\p_i + \psi_i\}_{i=1}^M$, where if $T_2$
is the analysis operator of $\Psi=\{\psi_i\}_{i=1}^M$, 
then $T_1^*T_2=0$.  That is, $\ran T_1 \perp \ran T_2$.
\end{theorem}

\begin{proof}
Note that the analysis operator of $\{\eta_i\}_{i=1}^M$ is $(T_1^*)^{-1}+T_2$.  Now,
\[ T_1^*\big((T_1^*)^{-1}+T_2\big) = T_1^*(T_1^*)^{-1}+T_1^*T_2 = Id+0 =Id.\]
By Proposition \ref{PAP1}, $\{\eta_i\}_{i=1}^M$ is a dual frame of $\Phi$.

Conversely, if $\{\eta_i\}_{i=1}^M$ is a dual frame for $\Phi$, let
$\psi_i = \eta_i-S^{-1}\p_i$, for all $i=1,2,\ldots,M$.  Then for all $x\in \H^N$,
\begin{align*}
T_1^* T_2 x &= \sum_{i=1}^M \langle x,\psi_i\rangle \p_i \\
&= \sum_{i=1}^M\langle x,\eta_i-S^{-1}\p_i\rangle \p_i\\
&= \sum_{i=1}^M \langle x,\eta_i\rangle \p_i-\sum_{i=1}^M \langle x,S^{-1}\p_i\rangle \p_i\\
&= x-x=0.
\end{align*}
This implies that for all $x,y\in \H^N$,
\[ \langle T_1x,T_2y\rangle = \langle x, T_1^*T_2 y \rangle = 0,\quad\]
which is precisely $\ran T_1 \perp \ran T_2$.
\end{proof}

\begin{proposition}
Let $\Phi=\{\p_i\}_{i=1}^M$ be a frame for $\H^N$ with frame operator $S$.  Then the only dual frame
of $\Phi$ which is isomorphic to $\Phi$ is $\{S^{-1}\p_i\}_{i=1}^M$.
\end{proposition}

\begin{proof}
Let $\{\psi_i\}_{i=1}^M$ be a dual frame for $\Phi$ and assume there is an invertible operator $F$
so that $\psi_i = FS^{-1}\p_i$ for all $i=1,2,\ldots,M$.  Then, for every $x\in \H^N$ we have
\begin{align*}
F^*x &= \sum_{i=1}^M\langle F^*x,S^{-1}\p_i\rangle\p_i\\
&= \sum_{i=1}^M \langle x,FS^{-1}\p_i\rangle \p_i\\
&= \sum_{i=1}^M\langle x,\psi_i\rangle\p_i\\
&= x.
\end{align*}
It follows that $F^*=Id$ and so $F=Id$.
\end{proof}

\subsection{Redundancy}

The main property of frames which makes them so useful in applied problems
is their {\bf \hl{redundancy}}.  That is, each vector in the space has infinitely
many representations with respect to the frame but it also has one natural
representation given by the frame coefficients.  The role played
by redundancy varies with specific applications.  One important role is its
{\bf \hl{robustness}}.  That is, by spreading our information over a wider range
of vectors, we are better able to sustain {\bf losses} (called 
{\bf \hl{erasures}} in this setting) and still have
accurate reconstruction.  This shows up in 
internet coding (for transmission losses), distributed processing
(where ``sensors'' are constantly fading out),
modeling the brain (where memory cells are constantly
dying out) and a host of other applications.  Another advantage of 
spreading our information over a wider range is to mitigate the effects
of noise in our signal or to make it prominent enough so it can be removed
as in signal/image processing.  A further upside of redundancy is in
areas such as quantum tomography where we need classes of orthonormal bases
which have ``constant'' interactions with one another or we need vectors
to form a Parseval frame but have the absolute values of their inner 
products with all other vectors the same.  In speech recognition, we need
a vector to be determined by the absolute value of its frame coefficients.
This is a very natural frame theory problem since this is impossible for
a linearly independent set to achieve.  Redundancy is a fundamental issue
in this setting.
\vskip12pt

Our next proposition shows the relationship between the frame elements
and the frame bounds.

\begin{proposition}\label{New}
Let $\{\p_{i}\}_{i=1}^M$ be a frame for $\mathbb{H}^N$ with frame bounds $A,B$.  
Then we have 
$\|\p_{i}\|^{2}\le B$ for all $1\le i\le M$, and if $\|\p_{i}\|^{2} = B$ holds for some $i$, then
 $\p_{i}\perp \spn\,\{\p_{j}\}_{j\not= i}$.  If $\|\p_{i}\|^{2}< A$,
then $\p_{i}\in \overline{\spn}\,\{\p_{j}\}_{j\not= i}$.
\end{proposition}

\begin{proof}
 If we replace $x$ in the frame definition by $\p_{i}$ we see that
$$
A\|\p_{i}\|^{2}\le \|\p_{i}\|^{4}+\sum_{j\not= i}|\langle\p_{i},\p_{j}\rangle|^{2} 
\le B\|\p_{i}\|^{2}.
$$
The first part of the result is now immediate.  For the second part, assume
to the contrary that $E = \overline{\spn}\,\{\p_{j}\}_{j\not= i}$ is a proper subspace
of $\H^N$.  Replacing $\p_{i}$ in the above inequality by $P_{E^{\perp}}\p_{i}$ and
using the left hand side of the inequality yields an immediate contradiction.  
\end{proof}

As a particular case of Proposition \ref{New} we have for a Parseval frame
$\{\p_{i}\}_{i=1}^M$ that $\|\p_{i}\|^{2}\le 1$ for all $i$, and $\|\p_{i}\|=1$  for some $i$ if and only if
$\p_{i}\perp \overline{\spn}\,\{\p_{j}\}_{j\not= i}$.  
We call $\{\p_i\}_{i=1}^M$ an {\bf {exact frame}\index{frame!exact}} if
it ceases to be a frame when any one of its vectors is
removed.  If $\{\p_{i}\}_{i=1}^M$ is an exact frame
then $\langle S^{-1}\p_{i},\p_{j}\rangle = \langle S^{-1/2}\p_{i},S^{-1/2}\p_{j}\rangle = {\delta}_{ij}$ 
(where ${\delta}_{ij}$ is the Kronecker delta) since $\{S^{-1/2}\p_{i}\}_{i=1}^M$
is now an orthonormal basis for $\H^N$.  That is, $\{S^{-1}\p_{i}\}_{i=1}^M$ and
$\{\p_{i}\}_{i=1}^M$ form a {\bf \hl{biorthogonal system}}.  Also, it follows that 
$\{e_{i}\}_{i=1}^N$
is an orthonormal basis for $\H^N$ if and only if it is an exact, Parseval
frame.  Another consequence of Proposition \ref{New} is the following.

\begin{proposition}
The removal of a vector from a frame leaves either a frame or an incomplete
set.
 \end{proposition}

\begin{proof}  By Theorem \ref{Tight}, we may assume that $\{\p_{i}\}_{i=1}^M$
 is a 
Parseval frame.  Now by Proposition \ref{New}, for any i, either
$\|\p_{i}\|=1$ and $\p_{i}\perp \spn\,\{\p_{j}\}_{j\not= i}$, or 
$\|\p_{i}\|< 1$ and $\p_{i}\in \overline{\spn}\,\{\p_{j}\}_{j\not= i}$. 
\end{proof}

\subsection{\hl{Minimal Moments}}

Since a frame is not independent (unless it is a Riesz basis)
a vector in the space will have many representations relative to the
frame besides the natural one given by the frame coefficients.  However,
the natural representation of a vector is the unique representation
of minimal ${\ell}_{2}$-norm as the following result of Duffin
and Schaeffer \cite{DS} shows.

\begin{theorem}\label{T2}
Let $\Phi = \{\p_{i}\}_{i=1}^M$ be a frame for a Hilbert space $\H^N$ with frame operator $S$ and $x\in \H^N$.  
If $\{b_{i}\}_{i=1}^M$ is any
sequence of scalars such that
$$
x=\sum_{i=1}^M b_{i}\p_{i},
$$
then
\begin{equation}\label{Moment}
\sum_{i=1}^M|b_{i}|^{2} = \sum_{i=1}^M|\langle S^{-1}x,\p_{i}\rangle|^{2} +
\sum_{i=1}^M|\langle S^{-1}x,\p_{i}\rangle-b_{i}|^{2}.
\end{equation}
\end{theorem}

\begin{proof}  We have by assumption
$$
\sum_{i=1}^M\langle S^{-1}x,\p_{i}\rangle \p_{i} =  x = \sum_{i=1}^Mb_{i}\p_{i}.
$$
Therefore, the sequence $\{\langle S^{-1}x,\p_{i}\rangle - b_i\}_{i =1}^M \in \ker T^* = [\ran T]^\perp$, where $T$ is the analysis operator of $\Phi$. Now, writing
\[
b_i = \langle S^{-1}x,\p_{i}\rangle - (\langle S^{-1}x,\p_{i}\rangle - b_i)
\]
and noting that  the sequence $\{\langle S^{-1}x,\p_{i}\rangle\}_{i = 1}^M \in \ran T$ and therefore perpendicular to $\{\langle S^{-1}x,\p_{i}\rangle - b_i\}_{i =1}^M$ gives (\ref{Moment}).
\end{proof}

\subsection{Orthogonal Projections and Naimark's Theorem}

A major advantage of frames over wavelets is that orthogonal
projections take frames to frames but do
not map wavelets to wavelets.

\begin{proposition}\label{P4}
Let $\{\p_{i}\}_{i=1}^M$ be a frame for $\H^N$ with frame bounds $A,B$,
 and let $P$
be an orthogonal projection on $H$.  Then $\{P\p_{i}\}_{i=1}^M$ is a frame for
$P(H)$ with frame bounds $A,B$.  In particular, an orthogonal
projection of a orthonormal basis (or a Parseval frame)
is a Parseval frame.  
\end{proposition}

\begin{proof}
For any $x\in P(H)$ we have
$$
\sum_{i=1}^M|\langle x,P\p_{i}\rangle|^{2} = \sum_{i=1}^M|\langle Px,\p_{i}\rangle|^{2} 
= \sum_{i=1}^M|\langle x,\p_{i}\rangle|^{2}.
$$
The result is now immediate. 
\end{proof}

Proposition \ref{P4} gives that an orthogonal projection $P$ applied an orthonormal basis $\{e_i\}_{i=1}^M$ for $\mathbb{H}^M$ leaves a Parseval frame $\{Pe_{i}\}_{i=1}^M$ for $P(\mathbb{H}^M)$.  The
converse of this is also true and is a result of Naimark (see \cite{CK}
and Han and Larson \cite{HL}).     

\begin{theorem}[\hl{Naimark's Theorem}]\label{HL}
A sequence $\{\p_{i}\}_{i=1}^M$ is a Parseval frame for a Hilbert space $\H^N$
if and only if there is a larger Hilbert space $\H^M \supset \H^N$ and an orthonormal
basis $\{e_{i}\}_{i=1}^M$ for $\H^M$ so that the orthogonal projection $P$ 
of $\H^M$
onto $\H^N$ satisfies $Pe_{i}=\p_{i}$ for all $i=1,2,\ldots,M.$
\end{theorem}

\begin{proof}
The ``if'' part follows from Proposition \ref{P4}.  For the ``only if'' part,
note that if $\{\p_{i}\}_{i=1}^M$ is a Parseval for $\H^N$,
then the synthesis operator
$T^{*}:{\ell}_{2}^M\rightarrow \H^N$ is a partial isometry.  Let $\{e_{i}\}_{i=1}^M$ be 
an orthonormal
basis for ${\ell}_{2}^M$ for which $T^{*}e_{i} = \p_{i}$ for all $i =1,2,\ldots,M$.   
Since the analysis operator $T$ is
 an isometry we can identify $\H^N$ with $T(\H^N)$. Now let 
$\H^M={\ell}_{2}^M$ and let $P$
be the orthogonal projection of $\H^M$ onto $T(\H^N)$.  Then for all 
$i=1,2,\ldots,M$ 
and all $y=Tx\in T(\H^N)$ we have
$$
\langle Tx,Pe_{i}\rangle = \langle PTx,e_{i}\rangle = \langle Tx,e_{i}\rangle = \langle x,T^{*}e_{i}\rangle = \langle x,\p_{i}\rangle
= \langle Tx,T\p_{i}\rangle.
$$
It follows that $Pe_{i} = T\p_{i}$, and the result follows from the association of $\H^N$ with
$T(\H^N)$.
\end{proof}

\begin{definition}
The {\bf standard $N$-\hl{simplex}} (or {\bf {\bf regular
$N$-gon}}) is
the subset of $\R^{N-1}$ given by 
unit norm $\{\p_i\}_{i=1}^{N}$ which are equiangular.
\end{definition}

It follows from Proposition \ref{P4}:

\begin{corollary}
 A regular simplex is an equiangular tight frame.
\end{corollary}

\begin{proof}
Given a regular $N$-simplex, it can be realized by
letting $\{e_i\}_{i=1}^N$ be an orthonormal basis for $\R^N$
and letting 
\[ x = \sum_{i=1}^N \varepsilon_i e_i,\mbox{ for any }\varepsilon_i
= 1, \mbox{ for }i=1,2,\ldots,N.\]
Now let $P$ be the orthogonal projection onto the span
of $x$, which
is given by:
\[ Py = \left [ \frac{1}{N}\sum_{i=1}^N \varepsilon_i \langle y,e_i\rangle
\right ] x\]
Then the $N$-simplex is a scaling by $1/\|(Id-P)e_i\|$ of $\{(Id-P)e_i\}_{i=1}^N$, which is
$N$-vectors in an $N-1$-dimensional Hilbert space and is
a Parseval frame by Proposition \ref{P4}.
\end{proof}

\subsection{Frame Representations}

Another important property of frames is given in the next proposition.

\begin{proposition}\label{presyn}
Let $\Phi = \{\p_i\}_{i=1}^M$ be a sequence of vectors in $\H^N$, let $\{e_j\}_{j=1}^N$ be
an orthonormal basis for $\H^N$, and let $\{\lambda_j\}_{j=1}^N$ be
positive real numbers.  The following are equivalent:
\begin{enumerate}
\item $\Phi$ is a frame for $\H^N$ with frame operator $S$
having eigenvectors $\{e_j\}_{j=1}^N$ and respective eigenvalues
$\{\lambda_j\}_{j=1}^N$.

\item  The following hold:
\begin{enumerate}
\item  If for all $j=1,2,\ldots ,N$,
\[ \psi_j := (\langle \p_1,e_j\rangle, \langle \p_2,e_j\rangle ,
\ldots , \langle \p_M,e_j\rangle),\]
then for all $1\leq j\not= k\leq N$ we have
\[ \langle \psi_j,\psi_k\rangle = 0.\]

\item For all $j=1,2,\ldots ,N$ we have $\|\psi_j\|_2^2
= \lambda_j$.
\end{enumerate}
\end{enumerate}
\end{proposition}

\begin{proof}
If $\{\p_i\}_{i=1}^M$ is a frame for $\H^N$ with frame operator $S$
having eigenvectors $\{e_j\}_{j=1}^N$ and respective eigenvalues
$\{\lambda_j\}_{j=1}^N$, then for all $j=1,2,\ldots ,N$ we have
\[ \sum_{i=1}^M \langle e_j,\p_i\rangle \p_i = \lambda_j e_j.\]
Hence, for $1 \leq j \neq k \leq N$ we have
\[ \langle \psi_j,\psi_k \rangle = \sum_{i=1}^M 
\overline{\langle \p_i,e_j\rangle} \langle \p_i,e_k\rangle  =\sum_{i=1}^M \langle e_j,\p_i\rangle \langle \p_i,e_k\rangle 
= \langle \lambda_j e_j , e_k \rangle = 0.\]
Similarly,
\begin{align*}
 \|\psi_j\|_2^2 &= \langle \lambda_j e_j , e_j \rangle = \lambda_j. \qedhere
\end{align*}
\end{proof}

This gives some important properties of the matrix representation of the synthesis operator. As stated before, these criterion are often used in constructing frames with a given synthesis operator.

\begin{proposition}\label{synprop}
If $\{\p_i\}_{i=1}^M$ is a frame for $\H^N$ and
$T^* = [a_{ij}]_{i=1,j=1}^{N\ ,M}$ is the synthesis matrix with respect to the eigenvectors of the frame operator, then the following hold.
\begin{enumerate}
\item The rows of $T^*$ are orthogonal.

\item The square sum of the columns are the square norm of the frame vectors.

\item The square sum of the rows are the eigenvalues of the frame operator.
\end{enumerate}
\end{proposition}

\section{Constants Related to Frames}\label{Intro}
\setcounter{equation}{0}

Let $\{\p_i\}_{i=1}^M$ be a frame for $\H^N$ with frame operator $S$. 

\smallskip
\noindent {\bf General Frame}:  The sum of the eigenvalues of $S$
equals the sum of the squares of the lengths of the frame vectors:
\[ \sum_{j=1}^N \lambda_j = \sum_{i=1}^M\|\p_i\|^2. \]
 \vskip12pt

\noindent {\bf Equal Norm Frame}:  For an equal norm frame in which
$\|\p_i\|=c$ holds for all $i=1,2,\ldots ,M$ we have
\[ \sum_{j=1}^N \lambda_j = \sum_{i=1}^M \|\p_i\|^2 = M \cdot c^2.\]
\vskip12pt

\noindent {\bf Tight Frame}:  Since tightness means $A=B$, we have
\[ \sum_{i=1}^M |\langle x,\p_i\rangle|^2 = A\|x\|^2,\ \ 
\mbox{for all $x\in \H^N$}.\]
We have that $S = A \cdot I_N$ and thus the sum of the eigenvalues becomes:
\[ N\cdot A = \sum_{j=1}^N \lambda_j = \sum_{i=1}^M \|\p_i\|^2.\]

\vskip12pt

\noindent {\bf Parseval Frame}:  If the frame is Parseval, then
$A=B=1$ and so
\[ \sum_{i=1}^M |\langle x,\p_i\rangle|^2 = \|x\|^2,\ \ 
\mbox{for all $x\in \HN$}.\]
We have that $S = Id$ and
\[ N = \sum_{j=1}^N \lambda_j = \sum_{i=1}^M \|\p_i\|^2.\]
\vskip12pt

\noindent {\bf Equal Norm Tight Frame}:  For an equal norm $A$-tight
frame in which
$\|\p_i\|=c$ holds for all $i=1,2,\ldots ,M$ we have
\[ N \cdot A = \sum_{j=1}^N \lambda_j = \sum_{i=1}^M \|\p_i\|^2 
= M \cdot c^2.\]
Hence $A = M \cdot c^2/N$ and thus
\[ \sum_{i=1}^M |\langle x,\p_i\rangle|^2 = \frac{M}{N}c^2\|x\|^2,\ \ 
\mbox{for all $x\in \H^N$.}\]

\vskip12pt

\noindent {\bf Equal Norm Parseval Frame}:  For an equal norm
Parseval frame we have
\[ N = \sum_{j=1}^M \lambda_j = \sum_{i=1}^M \|\p_i\|^2 = c^2M.\]
\vskip12pt

\section{Constructing Finite Frames}\label{Intro}
\setcounter{equation}{0}

For applications, we need to construct finite frames with extra
properties such as:
\begin{enumerate}
\item Prescribing in advance the norms of the frame vectors.
(See for example \cite{C2,CL,CL2}).
\vskip12pt

\item  Constructing equiangular frames.  That is, frames $\{\p_i\}_{i=1}^M$
for which there is a constant $c>0$ and 
\[ |\langle \p_i,\p_j\rangle| =c,\ \ \mbox{for all $i\not= j$}.\]
(See for example \cite{CRT,HP,SH,STDH}).
\vskip12pt

\item Frames for which the operator 
\[ \pm x \mapsto \{|\langle x,\p_i\rangle|\}_{i=1}^M,\ \ \mbox{is
one-to-one}.\]
(See for example \cite{BCE,BCEB}). 
\end{enumerate}

\noindent For a good introduction to constructive methods for frames see
\cite{C2} 

\subsection{Finding Parseval Frames}
There is a unique way to get Parseval frames of $M$ vectors
in $\H^N$.  Take a $M\times M$
unitary matrix $U = (a_{ij})_{i,j=1}^M$.  Take the submatrix
\[V = (a_{ij})_{i=1,j=1}^{M\ , N}.\]
The rows of $V$ form a parseval frame for $\H^N$, since these
vectors are just the rows of the matrix $U$ (which are an orthonormal
basis for $\H^M$) projected onto $\H^N$ and so form a Parseval
frame.  The converse to this is also true - and follows directly from
{\bf Naimark's Theorem} (Theorem \ref{HL}).

\subsection{Adding vectors to a frame to make it tight}

Every finite frame for $\HN$ can be turned into a tight frame with the
addition of at most $N-1$-vectors.

\begin{proposition}
If $\{\p_i\}_{i=1}^M$ is a frame for $\H^N$, then there are vectors
$\{\psi_j\}_{j=2}^N$ so that $\{\p_i\}_{i=1}^M \cup \{\psi_j\}_{j=2}^N$ is
a tight frame.
\end{proposition}

\begin{proof}
Let $S$ be the frame operator for the frame with eigenvectors
$\{e_j\}_{j=1}^N$ and respective eigenvalues $\{\lambda_j\}_{j=1}^N$
that satisfy $\lambda_1 \ge \lambda_2 \ge \cdots \ge \lambda_N$.
We define $\psi_j$ for $j=2,3,\ldots ,N$ by
\[ \psi_j = \sqrt{\lambda_1-\lambda_j}\,e_j.\]
This family $\{\p_i\}_{i=1}^M \cup \{\psi_j\}_{j=2}^N$ is a $\lambda_1$-tight frame.
\end{proof}

\subsection{Majorization}

One of the main constructive methods for frames is
due to Casazza and Leon \cite{CL,CL2}.  They gave a construction
for the important results of Benedetto and Fickus \cite{BF} and
Casazza, Fickus, Kova\v{c}evi\'c, Leon, Tremain \cite{CFK}.

\begin{theorem}\cite{CFK}\label{AT1}
Fix $N\le M$ and $a_1\ge a_2 \ge \cdots \ge a_M >0$.  The following
are equivalent:
\begin{enumerate}
\item There is a tight frame $\{\p_i\}_{i=1}^M$ for $\H^N$ satisfying
$\|\p_i\|=a_i$, for all $i=1,2,\ldots ,M$.

\item For all $1\le n < N$ we have
\[ a_n^2 \le \dfrac{\sum_{i=n+1}^Ma_i^2}{N-n}.\]

\item We have
\[ \sum_{i=1}^M a_i^2 \ge Na_1^2.\]

\item If
\[ \lambda =  \sqrt{\dfrac{N}{\sum_{i=1}^Ma_i^2}},\]
then
\[ \lambda a_i \le 1,\ \ \mbox{for all $i=1,2,\ldots ,M$}.\]
\end{enumerate}
\end{theorem}

This result was generalized by Casazza and Leon \cite{CL} to:

\begin{theorem}\label{tthm1}
Let $S$ be a positive self-adjoint operator on $\H^N$ and  let $\lambda_1 \ge \lambda_2
\ge \cdots \ge \lambda_N>0$ be the eigenvalues of $S$.  Fix $M\ge N$ and real numbers
$a_1\ge a_2\ge \cdots \ge a_M>0$.  The following are equivalent:
\begin{enumerate}
\item[(1)]  There is a frame $\{\p_i\}_{i=1}^M$ for $\H^N$ with frame operator $S$
satisfying $\|\p_i\|= a_i$ for all $i=1,2,\ldots,M$.
\item[(2)] For every $1\le k \le N$ we have
\[ \sum_{j=1}^ka_j^2 \le \sum_{j=1}^k \lambda_j,\]
and
\[ \sum_{i=1}^M a_i^2 = \sum_{j=1}^N \lambda_j.\]
\end{enumerate}
\end{theorem}

Theorem \ref{tthm1}(2) is the so called {\bfseries \hl{majorization}} of one sequence over another. The next result follows readily from the above results.

\begin{corollary}\label{AC20}
Let $ S $ be a positive self-adjoint operator on $\H^N$.  For any
$M\ge N$ there is an equal norm sequence $\{\p_i\}_{i=1}^{M}$
in $\H^N$ which has $S$ as its frame operator.
\end{corollary}

\begin{proof}
Let $ \lambda_1 \ge \lambda_2 \ge \dots \lambda_N > 0$ be the 
eigenvalues of $S$ and let
\begin{equation}\label{AE11}
a^2 = \frac{1}{M}\sum_{j=1}^N \lambda_j.
\end{equation}
Now we check condition (2) of Theorem \ref{tthm1} to see that there
is a sequence $\{\p_i\}_{i=1}^{M}$ in $\H^N$ with $\|\p_i\| = a$ for
all $i=1,2,\ldots ,M$.  That is, we check the condition with $a_1=a_2=\ldots a_M = a$.  
To check the equality in Theorem \ref{tthm1}(2), note that by Equation (\ref{AE11}) we have
\begin{equation}\label{AE12}
\sum_{i=1}^{M}a_i^2   = Ma^2 =  \sum_{i=1}^N \lambda_i.
\end{equation}
For the first inequality with $k = 1$ in Theorem \ref{tthm1}(2), we note that by
Equation (\ref{AE11}) we have that 
$$
a_1^2 = a^2 = \frac{1}{M}\sum_{i=1}^N \lambda_i 
\le \frac{1}{N}\sum_{i=1}^N \lambda_i
\le \lambda_1.
$$
So our inequality holds for $k=1$.  Suppose there is an $1< k\le N$ for
which this inequality fails and $k$ is the first time this fails.  So,
$$
\sum_{i=1}^{k-1}a_i^2 = (k-1)a^2 \le \sum_{i=1}^{k-1}\lambda_i,
$$
while
$$
\sum_{i=1}^{k}a_i^2 = ka^2 > \sum_{i=1}^{k}\lambda_i.
$$
It follows that
$$ 
a_{k+1}^2 = \cdots = a_N^2 = a^2 > \lambda_k \ge \lambda_{k+1} \geq \cdots \ge \lambda_N.
$$
Hence,
\begin{align*}
Ma^2 = \sum_{i=1}^{M}a_i^2  &\ge \sum_{i=1}^{k} a_i^2 + \sum_{i=k+1}^{N}a_i^2\\
&> \sum_{i=1}^{k}\lambda_i + \sum_{i=k+1}^{N}a_i^2 \\
&> \sum_{i=1}^{k} \lambda_i + \sum_{i=k+1}^{N}\lambda_i\\
&= \sum_{i=1}^{N}\lambda_i.
\end{align*}
But this contradicts Equation (\ref{AE12}).
\end{proof}

There was a recent {\bf significant advance} on this subject due to due to Cahill/Fickus/Mixon/Poteet/Strawn
\cite{CFMPS,FMPS} where they give an algorithm for constructing all self-adjoint matrices with prescribed
spectrum and diagonal and all finite frames with prescribed spectrum and
diagonal.  This work technically contains the solution to many of our frame questions.
That is, if we could carry out their construction with an additional restriction on it (e.g. requiring ``equiangular'' - see Subsection \ref{equiangular}) then we could construct equiangular tight frames.

\subsection{Spectral Tetris}

Another significant advance for frame theory came when {\bf \hl{Spectral Tetris}} was introduced by
Casazza/Fickus/Mixon/Wang/Zhou \cite{CFMWZ}.  This is now a massive subject and we refer the reader
to a comprehensive survey of Casazza/Woodland \cite{CW}.  We will just give an illustrative example here.

Before we begin our example, let us go over a few necessary facts for construction. Recall, that in order to construct an $M$-element unit norm tight frame (UNTF) in $\H^N$, we will construct an $N \times M$ synthesis matrix having the following properties:
\begin{enumerate}
\item The columns square sum to one, to obtain unit norm vectors.
\item The rows are orthogonal, which is equivalent to the frame operator, $S$, being a diagonal $N \times N$ matrix.
\item The rows have constant norm, to obtain tightness, meaning that $S=A\cdot Id$ for some constant $A$.
\end{enumerate}

\begin{remark} Since we will be constructing $M$-element UNTFs in $\H^N$, recall that the frame bound will be $A=\frac{M}{N}$. \end{remark}

Also, before the construction of a frame is possible, we must first ensure that such a frame exists by checking that the spectrum of the frame majorizes the square vector norms of the frame. However, this is not the only constraint. For Spectral Tetris to work, we also require that the frame has redundancy of at least 2, that is $M \geq 2N$, where $M$ is the number of frame elements and $N$ is the dimension of the Hilbert space. For a UNTF, since our unique eigenvalue is $\frac{M}{N}$, we see that this is equivalent to the requirement that the eigenvalue of the frame is greater than or equal to 2.  

The main idea of Spectral Tetris is to iteratively construct a synthesis matrix, $T^*$, for a UNTF one to two vectors at a time, which satisfies properties (1) and (2) at each step and gets closer to and eventually satisfies property (3) when complete. When it is necessary to build two vectors at a time throughout the Spectral Tetris process, we will utilize the following key $2 \times 2$ matrix as a building block for our construction.

Spectral Tetris relies on the existence of $2 \times 2$ matrices $A\left(x\right)$, for given $0\leq x \leq 2$, such that:
\begin{enumerate}
\item the columns of $A\left(x\right)$ square sum to $1$,
\item $A\left(x\right)$ has orthogonal rows,
\item the square sum of the first row is $x$.
\end{enumerate} 
These properties combined are equivalent to $$A\left(x\right)A^*\left(x\right)=\left[\begin{array}{cc}
x&0\vspace{.2cm}\\
0&2-x
\end{array}\right].$$
A matrix which satisfies these properties and which is used as a building block in Spectral Tetris is: 
$$A\left(x\right)=\left[\begin{array}{cc}
\sqrt{\frac{x}{2}}& \sqrt{\frac{x}{2}}\vspace{.2cm}\\
\sqrt{1-\frac{x}{2}}&-\sqrt{1-\frac{x}{2}}
\end{array}\right].
$$
We are now ready to give the example.

\begin{example}\label{stcex}
We would like to use Spectral Tetris to construct a sparse, unit norm, tight frame with 11 elements in $\H^4$, so our tight frame bound will be $\frac{11}{4}$.

To do this we will create a $4 \times 11$ matrix $T^*$, which satisfies the following conditions:
\begin{enumerate}
\item The columns square sum to $1$.
\item $T^*$ has orthogonal rows.
\item The rows square sum to $\frac{11}{4}$.
\item $S=T^*T=\frac{11}{4}\cdot Id$.
\end{enumerate}

Note that (4) follows if (1), (2) and (3) are all satisfied. Also notice that the sequence of eigenvalues $\{\lambda_j\}_{j=1}^4=\{\frac{11}{4},\frac{11}{4},\frac{11}{4},\frac{11}{4}\}$ majorizes the sequence of square norms $\{a_i^2\}_{i=1}^{11}=\{1,1,1,1,1,1,1,1,1,1,1\}$, which, in general, is necessary for such a frame to exist. 
 
Define $t_{i,j}$ to be the entry in the $i^{th}$ row and $j^{th}$ column of $T^*$. With an empty $4 \times 11$ matrix, we start at $t_{1,1}$ and work our way left to right to fill out the matrix. By requirement (1), we need the square sum of column one to be 1 and by requirement (2) we need the square sum of row one to be $\frac{11}{4} \geq 1$. Hence, we will start by being greedy and put the maximum weight of 1 in $t_{1,1}$. This forces the rest of the entries in column 1 to be zero, from requirement (1). We get:
$$T^*=\left[\begin{array}{ccccccccccc}
1&\cdot&\cdot&\cdot&\cdot&\cdot&\cdot&\cdot&\cdot&\cdot&\cdot\\
0&\cdot&\cdot&\cdot&\cdot&\cdot&\cdot&\cdot&\cdot&\cdot&\cdot\\
0&\cdot&\cdot&\cdot&\cdot&\cdot&\cdot&\cdot&\cdot&\cdot&\cdot\\
0&\cdot&\cdot&\cdot&\cdot&\cdot&\cdot&\cdot&\cdot&\cdot&\cdot
\end{array} \right].$$

Next, since row one needs to square sum to $\frac{11}{4}$, by (3), and we only have a total weight of 1 in row one, then we need to add $\frac{11}{4}-1=\frac{7}{4}=1+\frac{3}{4}\geq 1$ more weight to row one. So we will again be greedy and add another 1 in $t_{1,2}$. This forces the rest of the entries in column 2 to be zero, by (1). Also note that we have a total square sum of 2 in row one. We get:
$$T^*=\left[\begin{array}{ccccccccccc}
1&1&\cdot&\cdot&\cdot&\cdot&\cdot&\cdot&\cdot&\cdot&\cdot\\
0&0&\cdot&\cdot&\cdot&\cdot&\cdot&\cdot&\cdot&\cdot&\cdot\\
0&0&\cdot&\cdot&\cdot&\cdot&\cdot&\cdot&\cdot&\cdot&\cdot\\
0&0&\cdot&\cdot&\cdot&\cdot&\cdot&\cdot&\cdot&\cdot&\cdot
\end{array} \right].$$

In order to have a total square sum of $\frac{11}{4}$ in the first row, we need to add a total of $\frac{11}{4}-2=\frac{3}{4}<1$ more weight. If the remaining unknown entries are chosen so that $T^*$ has orthogonal rows, then $S$ will be a diagonal matrix. Currently, the diagonal entries of $S$ are mostly unknowns, having the form $\{2+?,\cdot,\cdot,\cdot\}$. Therefore we need a way to add $\frac{3}{4}$ more weight in the first row without compromising the orthogonality of the rows of $T^*$ nor the normality of its columns. That is, if we get ``greedy'' and try to add $\sqrt{\frac{3}{4}}$ to position $t_{1,3}$ then the rest of row one must be zero, yielding:

$$T^*=\left[\begin{array}{ccccccccccc}
1&1&\sqrt{\frac{3}{4}}&0&0&0&0&0&0&0&0\\
0&0&\cdot&\cdot&\cdot&\cdot&\cdot&\cdot&\cdot&\cdot&\cdot\\
0&0&\cdot&\cdot&\cdot&\cdot&\cdot&\cdot&\cdot&\cdot&\cdot\\
0&0&\cdot&\cdot&\cdot&\cdot&\cdot&\cdot&\cdot&\cdot&\cdot
\end{array} \right].$$

In order for column three to square sum to one, at least one of the entries $t_{2,3}, t_{3,3}$ or $t_{4,3}$ is non-zero. But then, it is impossible for the rows to be orthogonal and thus we cannot proceed. Hence, we need to instead add two columns of information in attempts to satisfy these conditions. The key idea is to utilize our $2 \times 2$ building block, $A\left(x\right)$, as defined at (a).

We define the third and fourth columns of $T^*$ according to such a matrix $A(x)$, where $x=\frac{11}{4}-2=\frac{3}{4}$. Notice that by doing this, column three and column four now square sum to one  within the first two rows, hence the rest of the unknown entries in these two columns will be zero. We get:
$$T^*=\left[\begin{array}{ccccccccccc}
1&1&\sqrt{\frac{3}{8}}&\sqrt{\frac{3}{8}}&\cdot&\cdot&\cdot&\cdot&\cdot&\cdot&\cdot\\
0&0&\sqrt{\frac{5}{8}}&-\sqrt{\frac{5}{8}}&\cdot&\cdot&\cdot&\cdot&\cdot&\cdot&\cdot\\
0&0&0&0&\cdot&\cdot&\cdot&\cdot&\cdot&\cdot&\cdot\\
0&0&0&0&\cdot&\cdot&\cdot&\cdot&\cdot&\cdot&\cdot
\end{array} \right].$$

The diagonal entries of $S$ are now $\{\frac{11}{4}, \frac{5}{4}+?,\cdot,\cdot\}$. The first row of $T^*$, and equivalently the first diagonal entry of $S$, now have sufficient weight and so its remaining entries are set to zero. The second row, however, is currently falling short by $\frac{11}{4}-\left(\left(\sqrt{\frac{5}{8}}\right)^2 + \left(-\sqrt{\frac{5}{8}}\right)^2\right)=\frac{6}{4}=1+\frac{2}{4}$. Since $1+\frac{2}{4} \geq 1$, we can be greedy and add a weight of 1 in $t_{2,5}$. Hence, column five becomes $e_2$. Next, with a weight of $\frac{2}{4}<1$ left to add to row two we utilize our $2 \times 2$ building block $A\left(x\right)$, with $x=\frac{2}{4}$. Adding this $2 \times 2$ block in columns six and seven yields sufficient weight in these columns and hence we finish these two columns with zeros. We get: 

$$T^*=\left[\begin{array}{ccccccccccc}
1&1&\sqrt{\frac{3}{8}}&\sqrt{\frac{3}{8}}&0&0&0&0&0&0&0\\
0&0&\sqrt{\frac{5}{8}}&-\sqrt{\frac{5}{8}}&1&\sqrt{\frac{2}{8}}&\sqrt{\frac{2}{8}}&0&0&0&0\\
0&0&0&0&0&\sqrt{\frac{6}{8}}&-\sqrt{\frac{2}{8}}&\cdot&\cdot&\cdot&\cdot\\
0&0&0&0&0&0&0&\cdot&\cdot&\cdot&\cdot
\end{array} \right].$$

The diagonal entries of $T^*$ are now $\{\frac{11}{4}, \frac{11}{4},\frac{6}{4}+?,\cdot\}$, where the third diagonal entry, and equivalently the third row, are falling short by $\frac{11}{4}-\frac{6}{4}=\frac{5}{4}=1+\frac{1}{4}$. Since $1+\frac{1}{4} \geq 1$, then we take the eighth column of $T^*$ to be $e_3$. We will complete our matrix following these same strategies, by letting the ninth and tenth columns arise from $A\left(\frac{1}{4}\right)$, and making the final column $e_4$, yielding the desired UNTF:

$$T^*=\left[\begin{array}{ccccccccccc}
1&1&\sqrt{\frac{3}{8}}&\sqrt{\frac{3}{8}}&0&0&0&0&0&0&0\\
0&0&\sqrt{\frac{5}{8}}&-\sqrt{\frac{5}{8}}&1&\sqrt{\frac{2}{8}}&\sqrt{\frac{2}{8}}&0&0&0&0\\
0&0&0&0&0&\sqrt{\frac{6}{8}}&-\sqrt{\frac{2}{8}}&1&\sqrt{\frac{7}{8}}&\sqrt{\frac{7}{8}}&0\\
0&0&0&0&0&0&0&0&\sqrt{\frac{7}{8}}&-\sqrt{\frac{7}{8}}&1
\end{array} \right].$$

\end{example}

In this construction, column vectors are either introduced one at a time, such as columns $1,2,5,8,$ and $11$, or in pairs, such as columns $\{3,4\}, \{6,7\},$ and $\{9,10\}$. Each singleton contributes a value of 1 to a particular diagonal entry of $S$, while each pair spreads two units of weight over two entries. Overall, we have formed a flat spectrum, $\{\frac{11}{4},\frac{11}{4},\frac{11}{4},\frac{11}{4}\}$, from blocks of area one or two. This construction is reminiscent of the game Tetris, as we fill in blocks of mixed area to obtain a flat spectrum.

\section{Gramian Operators}

If $\{\p_i\}_{i=1}^M$ is a frame for $\H^N$ with analysis operator $T$, the {\bf \hl{Gramian operator}\index{operator!Gramian}}
is defined as
\[ G :=TT^{*},\]
which has matrix representation
\[
G = \left [ \langle \p_j,\p_i\rangle \right ]_{i,j=1}^M,
\]
called the {\bfseries \hl{Gramian matrix}}. Since we know $T^*T$ and $TT^*$ have the same non-zero eigenvalues,
we have:
\begin{proposition}\label{sameev}
Let $\{\p_i\}_{i=1}^M$ be a frame for $\H^N$ with 
frame bounds $A,B$ and frame operator $S$.
The Gramian operator has the same non-zero eigenvalues as $S$.
That is, the largest eigenvalue of $G$ is less than or equal to $B$ and the smallest non-zero eigenvalue is greater than or equal to $A$.
\end{proposition}

\begin{theorem}
If $\{\p_i\}_{i=1}^M$ is a Parseval frame with analysis operator $T$,
then the Gramian operator is an orthogonal projection.
\end{theorem}

\begin{proof}
It is clear that $TT^*$ is self-adjoint and
\begin{align*} (TT^*)(TT^*)&= T(T^*T)T^* = T(I)T^* = TT^*. \qedhere \end{align*}
\end{proof}

\begin{corollary}
Let $\Phi =\{\p_i\}_{i=1}^M$ be vectors in $\H^N$.  The {Gramian} of $\Phi$ is invertible if and only if $\Phi$ is a 
Riesz basis  (that is, when $M = N$).
\end{corollary}

\begin{proof}
If $G = TT^{*}$ is invertible, by Proposition \ref{sameev}, $T^{*}T$ is
invertible.  Hence, $\{\p_i\}_{i=1}^M$ is a frame for $\H^N$.  Also,
$T^{*}$ is one-to-one, and further $T^{*}$ is bounded, linear
and onto.  Hence, it is an isomorphism.  

If $\{\p_i\}_{i=1}^M$ is a Riesz basis then $T^{*}$ is an
isomorphism and we have that $T^{*}$ is invertible and so $G = TT^{*}$
is invertible.
\end{proof}

\begin{proposition}
Let $F = [a_{ij}]_{i,j=1}^M$ be a positive, self-adjoint matrix
(operator) on $\H^N$ with $\dim\ker\,F=M-N$.  Then $\{F^{1/2}e_i\}_{i=1}^M$
spans an $N$-dimensional space and
\[ \langle F^{1/2}e_i,F^{1/2}e_j\rangle = \langle Fe_i,e_j\rangle
= a_{ij}.\]
Hence, $F$ is the Gramian matrix for the vectors $\{F^{1/2}e_i\}_{i=1}^M$.
Moreover, 
\[\|F^{1/2}e_i\|^2 = a_{ii},\]
and so if $a_{ii}=1$ for all $i=1,2,\ldots ,M$ then $\{F^{1/2}e_i\}_{i=1}^M$
is a unit norm family.
\end{proposition}

\section{Fusion Frames}\label{Intro}
\setcounter{equation}{0}

A number of new applications have emerged which cannot be modeled
naturally by one single frame system. Generally they share a common 
property that requires
distributed processing. Furthermore, we are often overwhelmed by a 
deluge of data assigned
to one single frame system, which becomes simply too large to be 
handled numerically. In these
cases it would be highly beneficial to split a large frame system 
into a set of (overlapping)
much smaller systems, and to process locally within each sub-system 
effectively.

A distributed frame theory for a set of local frame systems is therefore in 
demand.
A variety of applications require distributed processing.  Among them there 
are, for 
instance,
wireless sensor networks \cite{IB05}, geophones in geophysics measurements 
and studies \cite{CG06},
and the physiological structure of visual and hearing systems \cite{RJ05b}.
To understand the nature, the constraints, and related problems of these 
applications, let us elaborate
a bit further on the example of wireless sensor networks.

In wireless sensor networks, sensors of limited capacity and power are 
spread in an area sometimes as
large as an entire forest to measure the temperature, sound, vibration, 
pressure, motion, and/or pollutants.
In some applications, wireless sensors are placed in a geographical area 
to detect and characterize chemical,
biological, radiological, and nuclear material. Such a sensor system is 
typically redundant, and there is no
orthogonality among sensors, therefore each sensor functions as a frame 
element in the system.
Due to practical and cost reasons, most sensors employed in such 
applications have severe constraints in
their processing power and transmission bandwidth. They often have 
strictly metered power supply as well.
Consequently, a typical large sensor network necessarily divides the 
network into redundant
sub-networks -- forming a set of subspaces.
The primary goal is to have local measurements transmitted to a local 
sub-station within a subspace for
a subspace combining. An entire sensor system in such applications 
could have a number of such local
processing centers.  They function as relay stations, and have the 
gathered information further submitted
to a central processing station for final assembly.

In such applications, distributed/local processing is built in the 
problem formulation.
A staged processing structure is prescribed.  We will have to be able 
to process the
information stage by stage from local information and to eventually 
fuse them together
at the central station. We see therefore that a mechanism of coherently 
collecting
sub-station/subspace information is required.

Also, due to the often unpredictable nature of geographical factors, 
certain local
sensor systems are less reliable than others.   While facing the task of 
combining
local subspace information coherently, one has also to consider 
weighting the more
reliable sets of substation information more than suspected less 
reliable ones. Consequently,
the coherent combination mechanism we just saw as necessary often 
requires a weighted structure
as well. This all leads naturally to what is called a fusion frame.

\begin{definition}
Let $I$ be a countable index set, let $\{W_i\}_{i =1}^M$ be a family of closed subspaces
in $\H^N$, and let $\{v_i\}_{i =1}^M$ be a family of weights, i.e. $v_i > 0$ for all $i \in I$.
Then $\{(W_i,v_i)\}_{i=1}^M$ is a {\bfseries \hl{fusion frame}\index{frame!fusion}}, if there exist
constants $0 < C \le D < \infty$ such that
\begin{equation} \label{deffos}
C\|x\|^2 \le \sum_{i=1}^M v_i^2 \|{P}_{W_i}(x)\|^2 \le D\|x\|^2
\ \ \mbox{for all $x\in \H^N$},
\end{equation}
where ${P}_{W_i}$ is the orthogonal projection onto the subspace $W_i$.
We call $C$ and $D$ the {\bfseries fusion frame bounds\index{fusion frame!bounds}}.
The family $\{(W_i,v_i)\}_{i=1}^M$ is called a {\bfseries $C$-tight fusion frame\index{fusion frame!tight}}, if in (\ref{deffos})
the constants $C$ and $D$ can be chosen so that $C=D$, a {\bfseries Parseval fusion frame\index{fusion frame!Parseval}} provided that $C=D=1$,
and an {\bfseries \hl{orthonormal fusion basis}} if $\H^N = \bigoplus_{i=1}^M W_i$.
If $\{(W_i,v_i)\}_{i=1}^M$ possesses an upper fusion frame bound, but not necessarily a lower bound,
we call it a {\bfseries Bessel fusion sequence\index{Bessel!fusion sequence}} with {\bfseries Bessel fusion bound} $D$.
\end{definition}

Often it will become essential to consider a fusion frame together with a set
of local frames for its subspaces. In this case we speak of a {\rm fusion 
frame system}.

\begin{definition}
Let $\{(W_i,v_i)\}_{i=1}^M$  be a fusion frame for $\H^N$, and let 
the sequence of vectors $\{x_{ij}\}_{j =1}^{J_i}$
be a frame for $W_i$ for each $1 \leq i \leq M$. Then we call
$\{(W_i,v_i,\{\p_{ij}\}_{j =1}^{J_i})\}_{i=1}^M$ a {\bfseries {fusion frame system}\index{fusion frame!system}} 
for $\H^N$.
The constants, $C$ and $D$, are the associated {\bfseries {fusion frame bounds}} if they are the 
fusion frame bounds for
$\{(W_i,v_i)\}_{i=1}^M$, and $A$ and $B$ are the {\bfseries \hl{local frame bounds}\index{frame!local frame bounds}} 
if these
are the common frame bounds for the {\bfseries \hl{local frames}\index{frame!local}} 
$\{\p_{ij}\}_{j =1}^{ J_i}$ for each $1 \leq i \leq M$.
A collection of dual frames, $\{\psi_{ij}\}_{j =1}^{ J_i}$ for each $1 \leq i \leq M$,
associated with the local frames
will be called {\bfseries \hl{local dual frames}\index{dual frame!local}}.
\end{definition}

To provide a quick inside-look at some intriguing relations between properties
of the associated fusion frame and the sequence consisting of all local frame 
vectors,
we present the following theorem from \cite{CK04} that provides a link
between local and global properties. 

\begin{theorem}\cite[Thm. 3.2]{CK04} \label{theo:frame_local_global}
For each $1 \leq i \leq M$, let $v_i > 0$, let $W_i$ be a closed subspace
of $\H^N$, and let $\{\p_{ij}\}_{j =1}^{ J_i}$ be a frame for $W_i$
with frame bounds $A_i$ and $B_i$. Suppose
that
$0 < A = \inf_{1 \leq i \leq M} A_i \le \sup_{1 \leq i \leq M} B_i = B < \infty.$
Then the following conditions are equivalent.
\begin{enumerate}
\item $\{(W_i,v_i)\}_{i=1}^M$ is a fusion frame for $\H^N$.
\item $\{v_i \p_{ij}\}_{j =1,\, i = 1}^{J_i, \, M}$ is a frame for $\H^N$.
\end{enumerate}
In particular, if $\{(W_i,v_i,\{\p_{ij}\}_{j =1}^{J_i})\}_{i=1}^M$ is a 
fusion frame system for
$\H^N$ with fusion frame bounds $C$ and $D$, then $\{v_i \p_{ij}\}_{j 
=1,\, i =1}^{J_i, \, M}$ is a
frame for $\H^N$ with frame bounds $AC$ and $BD$.
Conversely, if $\{v_i \p_{ij}\}_{j =1, \, i = 1}^{J_i, \, M}$ is a frame for
$\H^N$ with frame bounds $C$ and $D$, then $\{(W_i,v_i,
\{\p_{ij}\}_{j=1}^{ J_i})\}_{i=1}^M$ is
a fusion frame system for $\H^N$ with fusion frame bounds 
$\frac{C}{B}$ and $\frac{D}{A}$.
\end{theorem}

Tight frames play a vital role in frame theory due to the fact that they
provide easy reconstruction formulas.  Tight fusion frames will turn
out to be particularly useful for distributed reconstruction as well.
Notice, that the previous theorem also implies that $\{(W_i,v_i)\}_{i=1}^{M}$ 
is a $C$-tight
fusion frame for $\H^N$ if and only if $\{v_i f_{ij}\}_{j=1,\, i=1}^{J_i, \, M}$ 
is a $C$-tight
frame for $\H^N$.

The following result 
from \cite{CKL} proves that the fusion frame bound $C$ of a $C$-tight
fusion
frame can be interpreted as the {\bfseries \hl{redundancy}} of this fusion frame.

\begin{proposition}
\label{prop:tight_redundancy} If $\{(W_i,v_i)\}_{i=1}^M$ is a
$C$-tight fusion frame for $\H^N$, then
\[ C = \frac{\sum_{i=1}^M v_i^2 \dim W_i}{N}.\]
\end{proposition}

Let $\cW = \{(W_i,v_i) \}_{i=1}^M$ be a fusion frame for $\H^N$. In order to map a signal
to the representation space, i.e., to analyze it, the {\bfseries \hl{fusion analysis operator}\index{analysis operator!fusion}\index{operator!fusion anaylsis}} $T_\cW$
is employed, which is defined by
\[ T_\cW:\H^N \rightarrow \Bigg(\sum_{i=1}^M\oplus W_i \Bigg)_{{\ell}_{2}}
\mbox{ with } \; T_\cW(x) = \big\{v_i P_{W_i}(x)\big\}_{i=1}^M.\]
It can easily be shown that the {\bfseries \hl{fusion synthesis operator}\index{synthesis operator!fusion}\index{operator!fusion synthesis}} $T_\cW^*$,
which is defined to be the adjoint operator of the analysis operator, is given by
\[T_\cW^*: \Bigg( \sum_{i=1}^M \oplus W_i \Bigg)_{{\ell}_{2}} \rightarrow \H^N\]
with
\[
T_\cW^*(x) = \sum_{i=1}^M v_i x_i, \mbox{ where }
x = \{x_i \}_{i=1}^M \in \Bigg( \sum_{i=1}^M \oplus W_i \Bigg)_{{\ell}_{2}}.\]
The {\bfseries \hl{fusion frame operator}\index{frame operator!fusion}} $S_\cW$ for $\cW$ is defined by
\[S_\cW(x) = T_\cW^*T_\cW(x)  = \sum_{i\in I}v_i^2 P_{W_i}(x).\]

Interestingly, a fusion frame operator exhibits properties similar to a frame operator
concerning invertibility.  In fact, if $\{(W_i,v_i)\}_{i=1}^M$ is a fusion frame for $\H^N$
with fusion frame bounds $C$ and $D$, then the associated fusion frame operator $S_\cW$ is
positive and invertible on $\H^N$, and
\begin{equation}
   C \cdot Id \le S_\cW \le D \cdot Id.  \label{eqn:CI_S_DI}
\end{equation}
We refer the reader to \cite[Prop. 3.16]{CK04} for details.

There has been a significant amount of recent work on fusion frames.
This topic now has its own website and we recommend visiting it
for the latest developments on fusion frames, distributed processing and
sensor networks. 
\vskip12pt

\begin{center}
{\bfseries http://www.fusionframe.org/}
\end{center}
{\bf }
\vskip12pt
But also visit the Frame Research Center Website:

\vskip12pt
\begin{center}
{\bfseries http://www.framerc.org/}
\end{center}
\section{Infinite Dimensional Hilbert Spaces}

We work with two standard infinite dimensional Hilbert Spaces.

\subsection{Hilbert Spaces of Sequences}

We being with the definition.

\begin{definition}
We define $\ell_2$ by:
\[ \{x = \{a_i\}_{i=1}^{\infty}:\|x\|:= \sum_{i=1}^{\infty}|a_i|^2 <\infty.\}\]
\end{definition}

The {\bf \hl{inner product}} of $x=\{a_i\}_{i=1}^{\infty}$ and $y=\{b_i\}_{i=1}^{\infty}$
is given by
\[ \langle x,y\rangle = \sum_{i=1}^{\infty}a_i\overline{b_i}.\]

The space $\ell_2$ has a natural orthonormal basis $\{e_i\}_{i=1}^{\infty}$
where $$e_i=(0,0,\ldots,0,1,0,\ldots)$$ where the $1$ is in the $i^{th}$-coordinate.
Most of the results on finite dimensional Hilbert spaces carry over here.  The one
major exception is that operators here may not have eigenvalues or be diagonalizable.
We give an example in the next subsection.  But all of the ``identities'' hold here.
Another difference is the notion of {\bf linear independence}.  

\begin{definition}
A family of vectors $\{x_i\}_{i=1}^{\infty}$ in $\ell_2$ is {\bf \hl{linearly independent}}
if for every finite subset $I\subset \N$ and any scalars $\{a_i\}_{i\in I}$ we have
\[ 
\sum_{i\in I}a_ix_i=0 \Rightarrow a_i=0,\mbox{ for all } i\in I.
\]
The family is $\omega$-independent if for any family of scalars $\{a_i\}_{i=1}^{\infty}$,
satisfying $\sum_{i=1}^{\infty}|a_i|^2 < \infty$,
we have
\[ \sum_{i=1}^{\infty}a_ix_i =0 \Rightarrow a_i =0 \mbox{ for all }i=1,2,\ldots.\]
\end{definition}

An orthonormal basis $\{e_i\}_{i=1}^{\infty}$ is clearly $\omega$-independent.
But, linearly independent vectors may not be $\omega$-independent.
For example, if we let
\[ x_i = e_i+ \frac{1}{2^i}e_{i+1},\ \ i=1,2,\ldots,\]
it is easily checked that this family is finitely linearly independent.  But,
\[ \sum_{i=1}^{\infty} \frac{(-1)^{i-1}}{2^i}x_i=0,\]
and so this family is not $\omega$-independent.

\subsection{Hilbert Spaces of Functions}

We define a Hilbert space of functions:

\begin{definition}
If $A\subset \R \,(\mbox{or }\C)$ we define
\[ L^2(A) =\bigg\{f:A \rightarrow \R \,(\mbox{or }\C):\|f\|^2 :=\int_{A}|f(t)|^2 dt < \infty\bigg\}.\]
The inner product of $f,g\in L^2(A)$ is
\[ \langle f,g\rangle = \int_{I}f(t)\overline{g(t)}dt.\]
\end{definition}
The two cases we work with the most are $L^2([0,1])$ and $L^2(\R)$.

The space $L^2([0,1])$ has a natural orthonormal basis given by the complex
exponentials:
\[ \{e^{2 \pi int}\}_{n\in \Z}.\]

If we choose $A \subset [0,1]$, the orthogonal projection of $L^2([0,1])$ onto
$L^2(A)$ satisfies:
\[ P(e^{2\pi int}) = \chi_A e^{2 \pi int}.\]
This family of vectors is a Parseval frame (since it is the image of an
orthonormal basis under an orthogonal projection) called the {\bf \hl{Fourier
frame}\index{frame!Fourier} on $A$}.  Recently, Marcus/Spielman/Srivastava \cite{MSS} solved
the {\bf Feichtinger Conjecture} for this class of frames (See subsection
\ref{feiconj}).

To get a flavor of things that don't go so nicely in the infinite dimensional setting, we note that there are positive, self-adjoint, invertible operators
on infinite dimensional Hilbert spaces which have no
eigenvectors. For example, consider the operator $S:L^2([0,1]) \to L^2([0,1])$ defined by $$S(f)(x) := (1+x)f(x).$$ For any
$f\in L^2[0,1]$ we have:
\[ \langle f,Sf\rangle = \int_0^1 (1+x)f^2(x)dx \ge\int_0^1f^2(x)dx = \|f\|^2.\]
So $S$ is a positive, self-adjoint, invertible operator.
However, in order for $Sf = \lambda f$ we would need to
have
\[ (1+x)f(x) = \lambda f(x) \mbox{ almost everywhere on } [0,1].\]
This is clearly impossible unless $f=0$, so $S$ has no eigenvectors.

\section{Major Open Problems in Frame Theory}
\setcounter{equation}{0}

In this section we look at some of the major open problems in Frame Theory.

\subsection{Equiangular Frames}\label{equiangular}

One of the simplest stated yet deepest problems in mathematics is the equiangular line
problem.

\begin{problem}
How many equiangular lines can be drawn through the origin in $\R^N$ or $\C^N$?
\end{problem}

The easiest way to describe {\bf \hl{equiangular lines}} is to put a unit norm vector
starting at the origin on each line, say $\{\p_i\}_{i=1}^M$, and the lines are
{\bf \hl{equiangular}} if there is a constant $0<c\le 1$ so that
\[ |\langle \p_i,\p_j\rangle | = c,\mbox{ for all }i\not= j.\]
    These inner products
represent the cosine of the acute angle between the lines.
The problem of constructing any number (especially, the maximal number)
of equiangular lines in ${\mathbb R}^N$ is one of the most elementary
and at the same time one of the most difficult problems in mathematics.
After sixty years of research, we do not know the answer for all dimensions
$\le 20$ in either the real or complex case.
  This line of
research was started in 1948 by Hanntjes \cite{H} in the setting of elliptic
geometry where he identified the maximal
number of equiangular lines in ${\mathbb R}^N$
for $n=2,3$.  Later, Van Lint and Seidel \cite{LiS} classified the
largest number of equiangular lines in ${\mathbb R}^N$ for dimensions
$N\le 7$ and at the same time emphasized the relations to
discrete mathematics.  
In 1973, Lemmens and Seidel \cite{LS} made a comprehensive
study of real equiangular line sets which is still today a fundamental
piece of work.  Gerzon \cite{LS} gave an upper bound for the maximal
number of equiangular lines in ${\mathbb R}^N$:

\begin{theorem}[Gerzon] \label{T2}If we have $M$ 
equiangular lines in ${\mathbb R}^N$ then
$$M\le \frac{N(N+1)}{2}.$$
\end{theorem}
It is known that we cannot reach this maximum in most cases.  
It is also known that the maximal number of equiangular lines in $\C^N$ is less
than or equal to $N^2$.  It is believed that this number of lines exists in $\C^N$ for
all $N$ but until now a positive answer does not exist for all $N\le 20$.  
\vskip8pt
  Also, P. Neumann \cite{LS} produced a fundamental result
in the area:

\begin{theorem}[P. Neumann]\label{T1}
If ${\mathbb R}^N$ has
$M$ equiangular lines at angle $1/\alpha$ and $M>2N$, then $\alpha$ is an odd
integer.  
\end{theorem}

Finally, there is a lower bound on the angle formed by equiangular
line sets.  
\begin{theorem}\label{T5}
If $\{\p_m\}_{m=1}^M$ is a set of norm one vectors in 
${\mathbb R}^N$, then
\begin{equation}\label{E1}
\max_{i\not= j}|\langle \p_i,\p_j\rangle| \ge \sqrt{\frac{M-N}{N(M-1)}}.
\end{equation}
Moreover, we have equality if and only if $\{\p_i\}_{i=1}^M$ is an
equiangular tight frame and in this case the tight frame bound
is $\frac{M}{N}$.      
\end{theorem}
\vskip8pt

This inequality goes back to Welch \cite{W}.  Strohmer
and Heath \cite{SH} and Holmes and Paulsen
\cite{HP} give more direct arguments which also yields
the ``moreover'' part.  For some reason, in the literature there is a further
assumption added to the ``moreover'' part of Theorem \ref{T5} that
the vectors span ${\mathbb R}^N$.  This assumption is not necessary.
That is, equality in inequality \ref{E1} already implies that the
vectors span the space \cite{CRT}.

\bigskip
\noindent {\bf Equiangular Tight Frames}:

Good references for real eqiuangular frames are \cite{CRT,HP,SH,STDH}.
A unit norm frame
with the property that there is a constant $c$ so that
$$
|\langle \p_i,\p_j\rangle| = c,\ \ \mbox{for all $i\not= j$},
$$
is called an {\bf \hl{equiangular frame}\index{frame!equiangular}} at angle $c$.
{Equiangular tight frames}\index{equiangular tight frames}\index{frame!equiangular tight} first appeared in discrete geometry \cite{LiS}
but today (especially the complex case)
have applications in signal processing, communications,
coding theory and more \cite{HSP, SH}.  A detailed study of this class
of frames was initiated by Strohmer and Heath \cite{SH}
and Holmes and Paulsen \cite{HP}.  
Holmes and Paulsen \cite{HP} showed that equiangular tight frames
give error correction codes that are robust against two erasures.
Bodmann and Paulsen \cite{BP} analyzed arbitrary numbers of
erasures for equiangular tight frames.  Recently, Bodmann,
Casazza, Edidin and Balan \cite{BCEB} showed that equiangular
tight frames are useful for signal reconstruction when all
phase information is lost.
Recently, Sustik, Tropp, Dhillon and Heath \cite{STDH} made an
important advance on this subject (and on the complex version).
Other applications
 include the construction of capacity achieving signature
sequences for multiuser communication systems in wireless communication
theory \cite{W}.  The tightness condition allows
equiangular tight frames to achiece the capacity of a Gaussian channel and their
equiangularity allows them to satisfy an interference invariance property.
Equiangular tight frames potentially have many 
more practical and theoretical applications.
Unfortunately, we know very few of them and so their usefulness
is largely untapped.  

The main problem:

\begin{problem}
Classify all equiangular tight frames, or find large classes of them.
\end{problem}

Fickus/Jasper/Mixon  \cite{FJM14} gave a large class of {\bf \hl{Kirkman equiangular tight
frames}} and used them in coding theory.

\begin{theorem}\label{theo1}
The following are equivalent:
\begin{enumerate}
\item The space ${\mathbb R}^N$ has an equiangular tight frame with $M$ elements
at angle $1/\alpha$.

\item We have
\[ M = \frac{(\alpha^2-1)N}{\alpha^2-N},\]
and there exist $M$ equiangular lines in ${\mathbb R}^N$ at angle $1/\alpha$.

\smallskip
\noindent Moreover, in this case we have:
\begin{enumerate}
\item $\alpha \le N \le \alpha^2-2$.

\item  $N=\alpha$ if and only if $M=N+1$.

\item  $N=\alpha^2-2$ if and only if  $M = \frac{N(N+1)}{2}$.

\item  $M=2N$ if and only if
\[ \alpha^2= 2N-1 = a^2+b^2,\ \ \mbox{a,b integers}.\]
\end{enumerate}

\noindent If $M\not= N+1,2N$ then:

\begin{enumerate}
\item[(e)]  $\alpha$ is an odd integer.

\item[(f)]  M is even.

\item[(g)] $\alpha$ divides M-1.

\item[(h)]  $\beta = \frac{M-1}{\alpha}$ is the angle for the
complementary equiangular tight frame.
\end{enumerate}
\end{enumerate}
\end{theorem}

\subsection{The Scaling Problem}

The \hl{scaling problem} is one of the deepest problems in frame theory.

\begin{problem}[Scaling Problem]
Classify the frames $\{\p_i\}_{i=1}^M$ for $\H^N$ so that there are scalars
$\{a_i\}_{i=1}^M$ for which $\{a_i\p_i\}_{i=1}^M$ is a Parseval frame?
Give an algorithm for finding $\{a_i\}_{i=1}^M$.
\end{problem}

This is really a special case of an even deeper problem.  

\begin{problem}
Given a frame $\{\p_i\}_{i=1}^M$ for $\H^N$, find the scalars
$\{a_i\}_{i=1}^M$ so that $\{a_i\p_i\}_{i=1}^M$ has the minimal
condition number with respect to all such scalings.
\end{problem}

For recent results on the scaling problem see
Chen/Kutyniok/Okoudjou/ Philipp/Wang \cite{XC},
Cahill/Chen \cite{CX} and Kutyniok/Okoudjou/Philipp/Tuley \cite{KOP}.

\subsection{Sparse Orthonormal Bases for Subspaces}

We will look at two questions concerning the construction of sparse
orthonormal bases.

\begin{definition}
Given a vector $x=(a_1,a_2,\ldots,a_N)\in \H^N$, we let
\[ \|x\|_0 = |\{1\le i \le N: a_i \not= 0\}.\]
\end{definition}

A natural question in Hilbert space theory is:

\begin{problem}\label{prob100}
Given a Hilbert space $\H^N$ with orthonormal basis $\{e_i\}_{i=1}^N$
and a $K$ dimensional subspace $W$, find the sparsest orthonormal basis for $W$
with respect to $\{e_i\}_{i=1}^N$.
That is, find a orthonormal basis $\{g_i\}_{i=1}^K$ for $W$ so that
\[ \sum_{i=1}^K\|g_i\|_0,\mbox{ is a minimum with respect to all orthonormal bases for W}.\]
\end{problem}

\noindent {\bfseries Sparse Gram-Schmidt Orthogonalization}:

There is a basic notion for turning a linearly independent set
into an orthonormal set with the same partial spans:
Gram-Schmidt Orthogonalization.  Given a set $\{\p_i\}_{i=1}^M$
of linearly independent vectors in $\H^N$, first let
\[ e_1 = \frac{\p_1}{\|\p_1\|}.\]
Assume we have constructed $\{e_i\}_{i=1}^K$ satisfying:
\begin{enumerate}
\item  $\{e_i\}_{i=1}^K$ is orthonormal.

\item We have
\[ \spn_{1\le i \le j}e_i = \spn_{1\le i \le j}\p_i,\mbox{ for all }j=1,2,\ldots,K.\]
\end{enumerate}
We then let
\[ \psi_{K+1} = \p_{K+1}-\sum_{i=1}^N \langle \p_{K+1},e_i\rangle e_i\]
and let
\[ e_{K+1}= \frac{\psi_{K+1}}{\|\psi_{K+1}\|}.\]

If we have a fixed basis $\{g_i\}_{i=1}^N$ for $\H^N$, we can compute
\[ \sum_{i=1}^K\|e_i\|_0\mbox{ with respect to the basis }\{g_i\}_{i=1}^N.\]
But this sum is different for different orderings of the original vectors
$\{\p_i\}_{i=1}^M$.
Related to Problem \ref{prob100} we have:
\begin{problem}
What is the correct ordering of $\{\p_i\}_{i=1}^M$ so that Gram-Schmidt
Orthogonalization produces the sparsest orthonormal sequence with
respect to all possible orderings?
\end{problem}

\subsection{The Paulsen Problem}

To state the \hl{Paulsen Problem}, we need some definitions.

\begin{definition}
A frame $\{\p_i\}_{i=1}^M$ for $\H^N$ with frame operator $S$
is said to be {\bf $\varepsilon$-{nearly equal norm}\index{frame!nearly equal norm}} if
\[ (1-\varepsilon)\frac{N}{M} \le \|\p_i\|^2 \le (1+\varepsilon)\frac{N}{M},
\mbox{ for all }i=1,2,\ldots,M,\]
and it is {\bf $\varepsilon$-{nearly Parseval}\index{frame!nearly Parseval}} if
\[ (1-\varepsilon) \cdot Id \le S \le (1+\varepsilon)\cdot Id.\]
\end{definition}

\begin{definition}
The {\bfseries \hl{distance between two frames}\index{frame!distance between}} $\Phi=\{\p_i\}_{i=1}^M$ and $\Psi = \{\psi_i\}_{i=1}^M$ is given by:
\[ d(\Phi,\Psi) = \sum_{i=1}^M\|\p_i-\psi_i\|^2.\]
\end{definition}
Because we did not take the square-root on the right-hand-side of the
above inequality, the function $d$ is not really a distance function.  

\bigskip
The {\bf Paulsen Problem} now states:
\begin{problem}
How close in terms of $d$ is an $\varepsilon$-nearly equal norm and $\varepsilon$-nearly
Parseval frame to an equal norm Parseval frame?
\end{problem}

The importance of this problem is that we have algorithms for finding
frames which are equal norm and nearly Parseval.  But we do not know
that these are actually close to any equal norm Parseval frame.

The closest equal norm frame to a frame is known and the closest
Parseval frame to a frame is known:

The closest equal norm frame to a frame $\{\p_i\}_{i=1}^M$ is
\[ \left \{ C \frac{\p_i}{\|\p_i\|}\right \}_{i=1}^M\mbox{ where } 
C := \frac{\sum_{i=1}^M\|\p_i\|}{M}.\]

If $\Phi=\{\p_i\}_{i=1}^M$ is a frame with frame operator $S$, the closest
Parseval frame to $\Phi$ is the canonical Parseval frame $\{S^{-1/2}\p_i\}_{i=1}^M$
\cite{CK07}
(see \cite{BC} for a better calculation). 

Also, there is an algorithm for turning a frame into an equal norm frame without
changing the frame operator \cite{C2}. 

Casazza/Cahill \cite{CC} showed that the Paulsen Problem is equivalent to an
old deep problem in Operator Theory.

\begin{problem}
Given a projection $P$ on $\H^N$ with ${\varepsilon}$-nearly equal diagonal
elements of its matrix, what is the closest constant diagonal projection to
$P$?
\end{problem}

\subsection{Concrete Construction of RIP Matrices}

{\bf \hl{Compressive Sensing}} is one of the most active areas of research
today. See the book \cite{FR} for an exhaustive coverage of this subject.
A fundamental tool in this area matrices with the {\bf {Restricted
Isometry Property}}, denoted {\bf RIP}.  Compressive sensing is a method
for solving {\bf underdetermined systems} if we have some form of
{\bf \hl{sparsity}} of the incoming signal.

\begin{definition}
A vector $x=(a_1,a_2,\ldots,a_N)\in \H^N$ is {\bf K-\hl{sparse}} if
\[ |\{1\le i \le N: a_i \not= 0\}|\le K.\]
\end{definition}  
The fundamental tool in compressive sensing is the class of {\bf \hl{Restricted
Isometry Property}} ({\bf RIP}) matrices.

\begin{definition}
A matrix $\Phi$ has the $(K,\delta)${\bf - Restricted
Isometry Property}, {\bf RIP} if
\[ (1-\delta)\|x\|^2 \le \|\Phi x\|^2 \le (1_+\delta)\|x\|^2,
\]
for every $K$-sparse vector $x$.  The smallest $\delta$
for which $\Phi$ is $(K,\delta)$-RIP is the {\bf restricted
isometry constant} {\bf (RIC)} $\delta_K$.
\end{definition}

The main result here is (see \cite{FR}):

\begin{theorem}
Given $\delta <1$, there exist $N \times M$ matrices with restricted isometry constant
$\delta_K \le \delta$ for 
\[ K \le c \frac{N}{\ln(N/K)},\]
for a universal constant $c$.
\end{theorem}

This means, for example, in an $N$-dimensional Hilbert space, we can find a set
of $100N$ norm one vectors $\{\p_i\}_{i=1}^{100N}$ for which every subset
of size $N/100$ is a $\delta$-Riesz basic sequence.  This is a quite amazing
result.  We know that in an $N$-dimensional Hilbert space, any orthogonal
set must have $\le N$ elements.  This result says that if we relax this requirement
just a little, we can find huge sets of vectors for which every subset of size a
proportion of the dimension of the space is nearly orthogonal.  

In the language of frame theory, we are looking for a family of unit norm vectors
$\Phi=\{\p_i\}_{i=1}^M$ in $\H^N$ so that every subset of $\Phi$ of size
a proportion of $N$ is nearly orthogonal and $M$ is much larger than $N$.
The existence of such matrices has been carried out by random matrix
theory.  Which means we do not know concretely a single such matrix,
despite the fact that these are essential for compressive sensing.
For years, the closest thing to {\bf concrete} here 
was a result of DeVore \cite{Dev} which constructed $N\times M$ matrices
for which subsets of size $\sqrt{N}$ were $\delta$-Riesz.  But this is far from
what we know is true which is subsets of size $cN$ for $0<c<1$ independent of
$N$.  Bandira/Fickus/Mixon/Wong \cite{BFMW} investigated various methods
for constructing RIP matrices.
Bourgain \cite{Bour} then {\bf broke the square root barrier} by showing we
can concretely construct RIP matrices with subsets of size $N^{1/2+\varepsilon}$
being $\delta$-Riesz. 
There is also an important result of Rudelson/Vershynin \cite{RV} which
says that if we take a random selection of rows from the Discrete Fourier Transform
Matrix,
then this submatrix will be a RIP matrix. 
  Since these matrices are fundamental to compressive
sensing, a longstanding, important and fundamental problem
here is:

\begin{problem}
Give a concrete construction of RIP matrices.
\end{problem}

\subsection{An Algorithm for the \hl{Feichtinger Conjecture}}\label{feiconj}

For nearly 50 years the Kadison-Singer problem \cite{KS} has
defied the best efforts of some of the most talented mathematicians
of our time.  It was just recently solved by Marcus/Spielman/Srivastava
\cite{MSS}.  For a good summary of the history of this problem and
consequences of this
achievement see \cite{C2014}.

In his work on time-frequency analysis, Feichtinger \cite{G,CT}
noted that all of the Gabor frames he was using 
had the property that they could be divided into a finite number
of subsets which were Riesz basic sequences.  This led to a 
conjecture known as the {\bf Feichtinger Conjecture} \cite{CT}.   
There is a significant body of work on this conjecture and we
refer the reader to \cite{C2014} for the best reference.

First we need:

\begin{definition}
A family of vectors $\{\p_i\}_{i\in I}$ is an $\varepsilon$-Riesz basic sequence
for $0<\varepsilon <1$ if for every family of scalars $\{a_i\}_{i\in I}$ we have
\[ (1-\varepsilon)\sum_{i\in I}|a_i|^2 \le \bigg\|\sum_{i\in I}a_i\p_i\bigg\|^2 \le
(1+\varepsilon)\sum_{i\in I}|a_i|^2.\]
\end{definition}

 The following theorem gives the best
quantative solution to
the {\bf Feichtinger Conjecture} from the results of \cite{MSS}.

\begin{theorem}[Marcus/Spielman/Srivastave]
Every unit norm $B$-Bessel sequence can be partitioned into
$r$-subsets each of which is a $\varepsilon$-Riesz basic sequence,
where
\[ r= \left ( \frac{6(B+1)}{\varepsilon}\right )^2 \mbox{ in the real case } ,\]
and
\[ r= \left ( \frac{6(B+1)}{\varepsilon}\right )^4 \mbox{ in the complex case } .\]
\end{theorem}

This theorem could be quite useful, except that it is an {\bf existence proof}.
Now what we really need is:

\begin{problem}
Find an implementable algorithm for proving the Feichtinger Conjecture.
\end{problem}

\subsection{Classifying Gabor Frames}

Gabor frames form the basis for time-frequency analysis which is
the mathematics behind signal processing.  This is a huge
subject which cannot be covered here except for a few remarks.  We
recommend the excellent book of Gr\"ochenig \cite{G1} for a comprehensive
coverage of this subject. We first define {\bfseries \hl{translation}} and {\bfseries \hl{modulation}}:

\begin{definition}
Fix $a,b >0$.
For a function $f\in L^2(\R)$ we define
\[ Translation \ \ by \ \ a:\ \ T_af(x) = f(x-a),\]
\[ Modulation \ \ by\ \  b:\ \ M_bf(x) = e^{2\pi ibx}f(x).\]
\end{definition}

In 1946, Gabor \cite{G} formulated a fundamental approach to signal
decomposition in terms of elementary signals.  Gabor's approach
quickly became a paradigm for the spectral analysis associated with
time-frequency methods, such as the short-time Fourier transform and
the Wigner transform.  For Gabor's method, we need to fix a
{\bf \hl{window function}} $g\in L^{\infty}(\R)$ and $a,b\in \R^{+}$.  If the
family
\[ \G(g,a,b)=\{ M_{mb}T_{na}g\}_{m,n\in \Z} \]
is a frame for $L^2(\R)$ we call this a {\bf \hl{Gabor frame}\index{frame!Gabor}}.  Gabor frames
are used in signal processing.  It is a
very deep question which values of $a,b,g$ give Gabor frames.  There
are some necessary requirements however.

\begin{theorem}
If the family given by $(g,a,b)$ yields a Gabor frame then:

(1)  $ab \le 1$.

(2)  If $ab =1$ then this family is a frame if and only if it
is a Riesz basis.  
\end{theorem}

Also, the Balian-Low Theorem puts some restrictions on the function
$g\in L^2(\R)$ for the case $ab=1$.

\begin{theorem}[\hl{Balian-Low Theorem}]
If $g\in L^2(\R)$, $ab=1$ and $(g,a,b)$ generates a
Gabor frame, then either $xg(x) \notin L^2(\R)$ or $g' \notin 
L^2(\R)$.
\end{theorem}

The Balian-Low Theorem implies that Gaussian functions $e^{-ax^2}$
cannot yield Gabor frames for $ab=1$.

The main problem here is:

\begin{problem}
Find all functions $g$ and positive constants $a,b$ so that $(g,a,b)$ forms a Gabor frame
for $L^2(\R)$.
\end{problem}

Recently, a significant advance was made on this problem by Dai/Sun
\cite{DS} when they solved the old and famous {\bf \hl{abc-problem}}.
We refer to \cite{DS} for the history of the problem.  In particular,
Dai/Sun classified all triples $(a,b,c)$ so that
\[ \G(\chi_I,a,b)\mbox{ is a Gabor frame  when } |I|=c.\]

\subsection{Phase Retrieval}

Phase retrieval is one of the largest areas of engineering with applications to
{\bf x-ray crystallography, Electron Microscopy, Coherence Theory, Diffractive
Imaging, Astronomical Imaging, x-ray tomography, Optics, Digital Holography, Speech
Recognition}
and more \cite{BM,BR,F78,F82,PDH,RJ,RBSC}.  
For an introduction to this subject see \cite{CCPW}.

{\bfseries Phase retrieval} is the problem of recovering a signal from the absolute values of linear measurement coefficients called {\bfseries intensity measurements}.  Note multiplying a signal by a global phase factor does not affect these coefficients, so we seek signal recovery mod a global phase factor.

There are two main approaches to this problem of phase retrieval.  One is to restrict the problem to a subclass of signals on which the intensity measurements become injective.  The other is to use a larger family of measurements so that the intensity measurements map any signal injectively.  The latter approach in phase retrieval first appears in \cite{BCE} where the authors examine injectivity of intensity measurements for finite Hilbert spaces.  The authors completely characterize measurement vectors in the real case which yield such injectivity, and they provide a surprisingly small upper bound on the minimal number of measurements required for the complex case.  This sparked an incredible volume of current phase retrieval research \cite{ABFM, BBCE, BCMN, CESV, CL, CSV, DH, EM} focused on algorithms and conditions guaranteeing injective and stable intensity measurements.

\begin{definition}
A family of vectors $\Phi=\{\p_i\}_{i=1}^M$ does {\bf \hl{phase retrieval}} on $\H^N$ if whenever $x,y \in\H^N$ satisfy
\[ |\langle x,\p_i\rangle| = |\langle y,\p_i\rangle|,\mbox{ for all }i=1,2\ldots,M,\]
then $x=cy$ where $|c|=1$.
\end{definition}

A fundamental result in phase retrieval involves the {\bf complement property}.

\begin{definition}
A family of vectors $\{\p_i\}_{i=1}^M$ in $\H^N$ has the {\bf \hl{complement property}} if
whenever we choose $I \subset \{1,2,\ldots,M\}$, at least one of the sets
$\{\p_i\}_{i\in I}$ or $\{\p_i\}_{i\in I^c}$ spans $\H^N$.
\end{definition}
Note that the complement property implies $M\ge 2N-1$.  For if $M\le 2N-2$ then we can
choose $I=1,2,\ldots,N-1$ and since the two induced subsets of our vectors each has only
$N-1$ vectors, neither can span $\H^N$.

The fundamental result here is due to Balan/Casazza/Edidin \cite{BCE}:

\begin{theorem}
In $\R^N$, a family of vectors $\{\phi_i\}_{i=1}^M$ does
phase retrieval if and only if it has the complement 
property.  Moreover,
there is a dense set of families of vectors $\{\p_i\}_{i=1}^{2N-1}$
which do phase retrieval.  
\end{theorem}

In the complex case, \cite{BCE} showed that a dense sent
of families of $(4N-2)$-vectors does phase retrieval.
Later, Bodmann \cite{Bod} showed that phase retrieval
can be done in the complex case with $(4N-4)$ vectors.
This was then improved  
 by Conca/Edidin/Hering/Vinzant \cite{CEHV}.

\begin{theorem}
In $\C^N$, there are families (in fact a dense set of families) of vectors $\{\p_i\}_{i=1}^{4N-4}$
which do phase retrieval.
\end{theorem}

Again, phase retrieval cannot be done with fewer vectors than $4N-4$.  

\bigskip
Given a signal $x$ in a Hilbert space, intensity measurements may also be thought of as norms of $x$ under rank one projections.  Here the spans of measurement vectors serve as the one dimensional range of the projections.  In some applications, however, a signal must be reconstructed from the norms of higher dimensional components.  In X-ray crystallography for example, such a problem arises with crystal twinning \cite{D}.  In this scenario, there exists a similar phase retrieval problem: given subspaces $\{W_n\}_{n=1}^M$ of an $N$-dimensional Hilbert space $\H^N$ and orthogonal projections $P_n:\H^N\rightarrow W_n$, can we recover any $x\in \H^N$ (up to a global phase factor) from the measurements $\{\norm{P_nx}\}_{n=1}^M$?  This problem was recently studied in \cite{BE} where the authors use semidefinite programming to develop a reconstruction algorithm for when the $\{W_n\}_{n=1}^M$ are equidimensional random subspaces.  Most results using random intensity measurements require the cardinality of measurements to scale linearly with the dimension of the signal space along with an additional logarithmic factor \cite{CSV}, but this logarithmic factor was recently removed in \cite{CL}.  Similarly, signal reconstruction from the norms of equidimensional random subspace components are possible with the cardinality of measurements scaling linearly with the dimension \cite{BE}.

In \cite{CCPW} it was shown:

\begin{theorem}
Phase retrieval can be done on $\R^N$ with $2N-1$ orthogonal projections of arbitrary rank.
\end{theorem}

This theorem raises an important question:

\begin{problem}
Can phase retrieval be done on $\R^N$ (respectively, $\C^N$) with fewer than $2N-1$ 
(respectively, $4N-4$) projections?  If so, what are the fewest number of projections
needed in both the real and complex case?
\end{problem}

\printindex


\begin{thebibliography}{WW}

\bibitem{ABFM} B. Alexeev, A. S. Bandeira, M. Fickus, D. G. Mixon, {\it Phase retrieval with polarization}, Available online: arXiv:1210.7752

\bibitem{BBCE} R. Balan, B. G. Bodmann, P. G. Casazza, D. Edidin, {\it Painless reconstruction from magnitudes of frame coefficients}, J. Fourier Anal. Appl. {\bf 15} (2009) 488-501.

\bibitem{BCE}  R. Balan, P.G. Casazza and D. Edidin,
\emph{On signal reconstruction without phase}, Appl. Comp. Harmonic
Anal. {\bf 20} (2006) 345-356.

\bibitem{BCMN}  A.S. Bandeira, J. Cahill, D.G. Mixon, and A.A. Nelson, {\em Saving phase: Injectivity and stability for phase retrieval},  arXiv:1302.4618v1.

\bibitem{BFMW}  A.S. Bandeira, M. Fickus, D. Mixon, and P. Wong, {\em The road to determiinistic
matrices with the restricted isometry property} arXiv:1202:1234v2.

\bibitem{BE}  C. Bachoc and M. Ehler, {\em Signal reconstruction from the magnitude of subspace components}, Available online: arXiv:1209.5986.

\bibitem{BR} C. Becchetti and L. P. Ricotti. {\it Speech recognition theory and C++ implementation}. Wiley (1999).

\bibitem{BF}  J.J. Benedetto and M. Fickus, \emph{Finite 
normalized tight frames}, Advances in computational Math., special
issue on frames (2002).

\bibitem{Bod}  B. Bodmann, {\it Stable phase retrieval
with low-redundancy frames}, Prepriint.

\bibitem{BC}  B. Bodmann and P.G. Casazza, {\em The Road to equal
norm Parseval frames}, Jour. Functional Analysis {\bf 258} No. 2 (2010)
397 - 420.

\bibitem{BCEB}  B. Bodmann, P.G. Casazza, D. Edidin and
R. Balan, \emph{Frames for linear reconstruction without phase},
Preprint.

\bibitem{BM} R. H. Bates and D. Mnyama. {\it The status of practical Fourier phase retrieval}, in W. H. Hawkes, ed., Advances in Electronics and Electron Physics, 67:1-64, 1986.

\bibitem{BP}  B. Bodmann and V.~Paulsen, \emph{Frames, graphs
and erasures}, Linear Algebra and Applications, {\bf 404} (2005) 118-146.

\bibitem{Bour}  J. Bourgain, S. Dilworth, K. Ford, S. Konyagin, and D. Kutzarova, {\em Explicit constructions
of RIP matrices and related problems}, Duke Math. Journal {\bf 159} No. 1 (2011) 145-185.

\bibitem{CX}  J. Cahill and X. Chen, {\em A note on scalable frames}, Proceedings of the 10th International
Conference on Sampling Theory and Applications, 93 - 96.


\bibitem{CC}  J. Cahill and P.G. Casazza, {\em The Paulsen Problem in Operator
Theory}, Operators and Matrices {\bf 7} No. 1 (2013) 116 - 130.

\bibitem{CCPW}  J. Cahill, P.G. Casazza, J. Peterson and L. Woodland,
{\bf Phase retrieval by projections}, prepring.

\bibitem{FMPS}  M. Fickus, D.G. Mixon, M.J. Poteet, and N. Strawn, {\em Constructing all
self-adjoint matrices with prescribed spectrum and diagonal}, Adv. Comput. Math. {\bf 39}
(2013) 585 - 609.

\bibitem{CFMPS}  J. Cahill, M. Fickus, D.G. Mixon, M.J. Poteet, and N. Strawn, {\em Constructing
finite frames with a given spectrum and set of lengths}, Appl. Comput. harmon.
Anal. {\bf 35} (2013) 52 - 73.

\bibitem{CESV} E. J. Cand`es, Y. Eldar, T. Strohmer, V. Voroninski, {\it Phase retrieval via matrix completion}, Available online: arXiv:1109.0573.

\bibitem{CL} E. J. Cand`es, X. Li, {\it Solving quadratic equations via PhaseLift when there are about as many equations as unknowns}, Available online: arXiv:1208.6247.

\bibitem{CSV} E.J. Candes, T. Strohmer, V. Voroninski, {\em PhaseLift: Exact and stable signal recovery from magnitude measurements via convex programming}, Available online: arXiv:1109.4499.


\bibitem{CA}  P.G.~Casazza, {\em The Art of Frame Theory}, Taiwanese
Journal of Math. {\bf 4}, No. 2 (2000) 1-127.

\bibitem{C1} P.G.~Casazza, {\em Modern tools for Weyl-Heisenberg (Gabor)
frame theory}, Advances in Imaging and Electron Physics {\bf 115}
(2000) 1-127.

\bibitem{C2}  P.G.~Casazza, \emph{Custom building finite frames},
Contemp Math. {\bf 345} (2004) 61-86.

\bibitem{C2014}  P.G. Casazza, {\em Consequences of the Marcus/Spielman/Srivastava
solution to the Kadison-Singer Problem}, Preprint.

\bibitem{CFMWZ}  P.G. Casazza, M. Fickus, D. Mixon, Y. Wang, and Z. Zhou,
{\em Constructing tight fusion frames},Appl. Comput. Harmon. Anal. {\bf 30}
(2011) 175 - 187.


\bibitem{C} P.G. Casazza, M. Fickus, J.C. Tremain and E. Weber, {\em The
    Kadison-Singer Problem in Mathematics and Engineering: A Detailed
    Account,}
Operator Theory,
Operator Algebras and Applications, Proceedings of the 25th GPOTS
Symposium (2005), D. Han, P.E.T. Jorgensen and D.R. Larson Eds.,
Contemporary Math {\bf 414}(2006) 299-356. 

\bibitem{CFK}  P.G. Casazza,J. Kova\v{c}evi\'c, M. Fickus, M. Leon
and J.C. Tremain, \emph{A physical interpretation for finite tight
frames}, Appl. Comp. Harmonic Anal. (2006) 51-78. 

\bibitem{CK} P.G. Casazza and J. Kova\v{c}evi\'c, {\em Equal norm
tight frames with erasures}, Adv. Comp. Math {\bf 18} (2003) 387-430.

\bibitem{CK04}
P. G. Casazza and G. Kutyniok,
{\em Frames of subspaces},
in: Wavelets, Frames and Operator Theory (College Park, MD, 2003),
Contemp. Math. {\bf 345}, Amer. Math. Soc., Providence, RI, 2004,
87--113.

\bibitem{CK12}  P.G. Casazza and G. Kutyniok Eds. {\em Finite Frames:
Theory and Applications}, Birkhauser, Boston (2012).

\bibitem{CKL}  P.G. Casazza, G. Kutyniok and S. Li, \emph{Fusion frames
and distributed processing}, Preprint.


\bibitem{CK07}
P. G. Casazza and G. Kutyniok,
{\em A generalization of Gram-Schmidt orthogonalization generating
all Parseval frames,}
Adv. Comput. Math.  {\bf 27}  (2007), 65--78.


\bibitem{CL}  P.G. Casazza and M. Leon, \emph{Existence and construction
of finite tight frames}, Preprint.

\bibitem{CL2} P.G. Casazza and M. Leon, \emph{Existence and construction
of finite frames with a given frame operator}, International Journal of Pure and
Applied Mathematics, {\bf 63} No. 2 (2010) 149-158.

\bibitem{CRT} P.G. Casazza, D. Redmond and J.C. Tremain, \emph{
Real equiangular frames}, Preprint.

\bibitem{CT} P.G. Casazza and J.C. Tremain, {\em The Kadison-Singer
    Problem in Mathematics and Engineering,} 
Proceedings of the National
Academy of Sciences, {\bf 103} No. 7 (2006) 2032-2039.

\bibitem{CW}  P.G. Casazza and L. Woodland, {\em The fundamentals of spectral
tetris frame constructions}, Preprint.

\bibitem{XC}  X. Chen, G. Kutyniok, K.A. Okoudjou, F. Philipp, and R. Wang,
{\bf Measures of scalability}, preprint.

\bibitem{Ch}  O.~Christensen, \emph{An introduction to frames and
Riesz bases}, Birkh\"auser, Boston 2003.

\bibitem{CEHV}  A. Conco, D. Edidin, M. Hering, and C. Vinzant, {\em An algebraic characterization
of injectivity in phase retrieval}, arXiv:1312:0158v1.

\bibitem{CG06}
M. S. Craig and R. L. Genter,
{\em Geophone array formation and semblance evaluation},
Geophysics {\bf 71} (2006), 1--8.

\bibitem{DS}  X.R. Dai and Q. Sun, {\em The abc-problem for Gabor Systems},
arXiv:1304.7750v1.

\bibitem{Dev}  R.A. DeVore, {\em Deterministic conctructions of compressed sending matrices},
J. Convexity {\bf 23} (2007) 918 -  925.

\bibitem{DH} L. Demanet, P. Hand, {\it Stable optimizationless recovery from phaseless linear measurements}, Available online: arXiv:1208.1803.

\bibitem{D} J. Drenth, {\it Principles of protein x-ray crystallography}, Springer, 2010.



\bibitem{DS}  R.J. Duffin and A.C. Schaeffer, 
{\em A class of nonharmonic Fourier series.} 
Trans. AMS 72 (1952) 341-366.

\bibitem{EM} Y. C. {\it Eldar, S. Mendelson, Phase retrieval: Stability and recovery guarantees}, Available online: arXiv:1211.0872

\bibitem{FJM14}  M. Fickus, J. Jasper and D. Mixon, {\em Kirkman equiangular tight
frames and codes}, IEEE Trans. Inform. Theory, {\bf 60} (2014) 170 - 181.

\bibitem{F78} J. R. Fienup. {\it Reconstruction of an object from the modulus of its fourier transform}, Optics Letters, {\bf 3} (1978), 27-29.

\bibitem{F82} J. R. Fienup. {\it Phase retrieval algorithms: A comparison}, Applied Optics, {\bf 21} (15) (1982), 2758-2768.

\bibitem{FRC} {\it Frame Research Center}, http://www.framerc.org/

\bibitem{G}  D. Gabor, \emph{Theory of Communications},
J. Inst. Elec. Engrg. {\bf 93} (1946) 429-457.

\bibitem{G1}  K.H. Gr\"ochenig, {\em Foundations of time-frequency
analysis}, Birkh\"auser, Boston, 2000.


\bibitem{FR}  S. Foucart and H. Rauhut, {\em A Mathematical Introduction
to Compressive Sensing}, Birkauser, Boston (2013).

\bibitem{HSP}
R.W.~Heath, T. Strohmer and A. J. Paulraj,
\emph{On quasi-orthogonal signatures for CDMA}, IEEE Trans.
Information Theory {\bf 52} No. 3 (2006) 1217-1226.

\bibitem{H}  J.~Haantjes, \emph{Equilateral point-sets in elliptic two-
and three-dimensional spaces}, Nieuw Arch. Wisk. {\bf 22} (1948)
355-362.

\bibitem{HKLW}  D. Han, K. Kornelson, D. Larson, and E. Weber,
{\em Frames for Undergraduates}, Student Mathematical Library,
AMS {\bf 40} (2007).

\bibitem{HL}  D. Han and D.R. Larson, {\em Frames, bases and
group representations}, Memoirs AMS {\bf 697} (2000).



\bibitem{HP}  R. B.~Holmes and V. I.~Paulsen, \emph{Optimal frames for
erasures}, Linear Alg. and Applications {\bf 377} (204) 31-51.

\bibitem{IB05}
S. S. Iyengar and R. R. Brooks, eds.,
{\em Distributed Sensor Networks},
Chapman \& Hall/CRC, Baton Rouge, 2005.

\bibitem{KS} R. Kadison and I. Singer,
{\em Extensions of pure states}, American Jour. Math. {\bf 81}
(1959), 383--400.


\bibitem{K}  J. Kova\v{c}evi\'{c}and A. Chebira, {\em An introduction
to frames}, in {\em Foundations and Trends in Signal Processing}
(2008) NOW publishers.

\bibitem{KOP}  G. Kutyniok, K.A. Okoudjou, and F. Philipp, {\em Scalable frames and convex
geometry}, Spectra of wavelets, tilings and frames (Boulder, 2012), Contemp. Math.
{\bf 345} AMS (2013).



\bibitem{LS}  P.~Lemmens and J.~Seidel, \emph{Equiangular lines},
Journal of Algebra {\bf 24} (1973) 494-512.

\bibitem{Li95}
S. Li,
{\em On general frame decompositions},
Numer. Funct. Anal. Optim. {\bf 16}  (1995), 1181--1191.

\bibitem{LiS}  J. H.~van Lint and J. J.~Seidel, \emph{Equiangular point
sets in elliptic geometry}, Proc. Nederl. Akad. Wetensch. Series A {\bf 69}
(1966) 335-348.

\bibitem{MSS}  A. Marcus, D. Spielman and N. Srivastava, {\it Interlacing
families II:  Mixed characteristic polynomials and the Kadison-Singer
Problem}, arXiv 1306.3969v4.



\bibitem{PDH} J. G. Proakis, J. R. Deller and J. H. L. Hansen. {\it Discrete-Time processing of speech signals}. IEEE Press
(2000).


\bibitem{RJ} L. Rabiner and B. H. Juang. {\it Fundamentals of speech recognition}. Prentice Hall Signal Processing Series
(1993).

\bibitem{RBSC} J. M. Renes, R. Blume-Kohout, A. J. Scott, and C. M. Caves, {\it Symmetric Informationally Complete Quantum Measurements}, J. Math. Phys., {\bf 45}, pp. 2171-2180, 2004.



\bibitem{RJ05b}
C. J. Rozell and D. H. Johnson,
\emph{Analyzing the robustness of redundant population codes in sonsory
and feature extraction systems}, Neurocomputing {\bf 69} (2006), 1215--1218.

\bibitem{RV}  M. Rudelson and R. Vershynin, {\em On sparse reconstruction from Fourier
and Gaussian measurements}, Comm. Pure Appl. Math. {\bf 61} (2008) 1025 - 1045.

\bibitem{R}  W. Rudin, {\em Functional Analysis}, McGraw Hill (1991).



\bibitem{SH} T.~Strohmer and R. W.~Heath, \emph{Grassmannian frames
with applications to coding and communication}, Appl. Comp. Harmonic
Anal. {\bf 14} No. 3 (2003) 257-275.

\bibitem{STDH}  M. A.~Sustik, J. A.~Tropp, I. S.~Dhillon and R. W.~
Heath, Jr., \emph{On the existence of equiangular tight frames},
Linear Alg. and Applications {\bf 426} No. 2-3 (2007) 619-635.

\bibitem{T} J.C.~Tremain, \emph{Concrete Constructions of Equiangular
Line Sets}, in preparation.



\bibitem{W}  L. R.~Welch, \emph{Lower bounds on the maximum cross-correlation
of signals}, IEEE Trans. Inform. Theory {\bf 20} (1974) 397-399.

\end{thebibliography}
\end{document}